\documentclass[11pt,a4paper]{article}
\title{Product and Coproduct on Fixed Point Floer Homology of Positive Dehn Twists}
\author{Yuan Yao, Ziwen Zhao}
\date{}
\usepackage{fullpage}
\usepackage{xcolor}
\usepackage{soul}
\usepackage{mathtools}
\usepackage{graphicx}
\usepackage{tikz-cd}
\usepackage{graphicx}
\usepackage{amsmath}
\DeclareMathOperator{\CZ}{CZ}
\DeclareMathOperator{\ind}{ind}
\DeclareMathOperator{\Ver}{Ver}
\DeclareMathOperator{\proj}{proj}
\DeclareMathOperator{\diag}{diag}
\DeclareMathOperator{\Ker}{Ker}
\usepackage{amssymb}
\usepackage{amsfonts}
\usepackage{amsthm}
\usepackage[utf8]{inputenc}
\usepackage[backend=biber,style=alphabetic,giveninits=true,sorting=nyt,maxbibnames=99]{biblatex}
\DeclareFieldFormat
  [article,inbook,incollection,inproceedings,patent,thesis,unpublished]
  {title}{#1\isdot}
\DeclareFieldFormat[article,inbook,incollection,inproceedings,patent,thesis,unpublished]{title}{\textit{#1}}
\DeclareFieldFormat[article,inbook,incollection,inproceedings,patent,thesis,unpublished]{journaltitle}{#1}
\DeclareFieldFormat[article,inbook,incollection,inproceedings,patent,thesis,unpublished]{volume}{\textbf{#1}}
\DeclareFieldFormat{pages}{#1}
\renewbibmacro{in:}{}

\usepackage{hyperref}
\theoremstyle{plain}
\newtheorem{theorem}{Theorem}[section]
\newtheorem{lemma}[theorem]{Lemma}
\newtheorem{cor}[theorem]{Corollary}
\newtheorem{proposition}[theorem]{Proposition}
\newtheorem{definition}[theorem]{Definition}
\theoremstyle{remark}
\newtheorem{remark}[theorem]{Remark}
\newcommand{\R}{\mathbb{R}}

\newcommand{\Z}{\mathbb{Z}}
\begin{document}

\maketitle
\begin{abstract}
    We compute the product and coproduct structures on the fixed point Floer homology of iterations on the single Dehn twist, subject to some mild topological restrictions. We show that the resulting product and coproduct structures are determined by the product and coproduct on Morse homology of the complement of the twist region, together with certain sectors of product and coproduct structures on the symplectic homology of $T^*S^1$. The computation is done via a direct enumeration of $J-$holomorphic sections: we use a local energy inequality to show that some of the putative holomorphic sections do not exist, and we use a gluing construction plus some Morse-Bott theory to construct the sections we could not rule out.
\end{abstract}
\tableofcontents

\section{Introduction}
In this paper, we calculate the product and coproduct structures on fixed point Floer homologies of iterations of a single positive Dehn twist on a surface.

\subsection{Fixed point Floer homology}
In this section, we briefly review the definition of the fixed point Floer homology of a symplectomorphism on a compact surface. Let $(\Sigma,\omega_0)$ be a compact symplectic surface (possibly with boundary) and $\phi:\Sigma\to \Sigma$ a symplectomorphism. If $\partial \Sigma$ is nonempty, we further assume that near each boundary component of $\partial \Sigma$, we can identify an open neighborhood with\footnote{Here and throughout the paper, we use notations like $S^1_y$ and $[0,1]_t$ to indicate the coordinates we choose. For example $S^1_y$ denotes the circle $S^1$ with coordinates written as $y\in S^1$. Later on we will use the notation $[S^1_y]$ to denote the homology class generated by this circle whose coordinates we have chosen to be $y \in S^1$.} $((-\epsilon_i,0]_{x_i}\times S^1_{y_i}, dx_i\wedge dy_i)$, such that $\phi$ is the time-1 map of the Hamiltonian $H_i(x_i, y_i) = \theta_i x_i$ for a small irrational number $\theta_i$. Notice that under such assumptions, there are no fixed points of $\phi$ near $\partial \Sigma$. We assume that 
$\phi$ is nondegenerate, that is, for every fixed point $x$ of $\phi$, the linearization $d\phi_x$ does not have $1$ as an eigenvalue. The fixed point Floer homology is the homology of the chain complex ($CF_*(\Sigma, \phi)$, $\partial$), whose underlying module is generated over $\Z/2$ by all fixed points of $\phi$. 

The differential of $CF_*(\Sigma, \phi)$ is defined by counting $J-$holomorphic cylinders in the symplectization of the mapping torus of $\phi$. More precisely, for any symplectomorphism $\phi:\Sigma\to \Sigma$, the mapping torus $Y_\phi$ is defined as
\begin{equation}
    Y_\phi = [0,1]_t\times \Sigma/((1,p)\sim(0,\phi(p))).
\end{equation}

The mapping torus comes with a projection $\pi: Y_\phi\to S^1_t$, and the symplectic form $\omega_0$ induces a closed 2-form $\omega_\phi\in\Omega^2(Y_\phi)$ which restricts to $\omega_0$ on each fiber (to be more precise, $\omega_0$ pulls back to a closed 2-form on the product $[0,1]\times \Sigma$, and $\omega_\phi$ is the induced 2-form on the quotient space $Y_\phi$). The vector field $\partial_t$ on $[0,1]_t\times \Sigma$ descends to a vector field on $Y_\phi$, which we still denote by $\partial_t$. Notice that there is a one-to-one correspondence between fixed points of $\phi$ and closed orbits of $\partial_t$ that cover $S^1_t$ once. We will denote by $\gamma_x$ the closed orbit associated to a fixed point $x$ .

The projection $\pi$ now extends to $\R\times Y_\phi\to \R\times S^1$, and the fiberwise symplectic form $\omega_\phi$ extends to the symplectization as well. The symplectization of $Y_\phi$ is the 4-manifold $\R_s \times Y_\phi$ together with the symplectic form $ds\wedge dt+\omega_\phi$. An almost complex structure $J$ on $\R\times Y_\phi$ is called \emph{$\phi$-compatible}, if it is invariant under the natural $\R$-action, sends $\partial_s$ to $\partial_t$, sends $\text{ker } d\pi$ to itself, and that $\omega_\phi(\cdot, J\cdot)$ is a Riemannian metric on $\text{ker } d\pi$. Given 2 fixed points $x$, $y$ of $\phi$, we define the moduli space $\mathcal{M}^J_{x,y}$ to be
\begin{equation}
    \mathcal{M}^J_{x,y}:=\{u: \R_s \times S^1_t \to \R\times Y_\phi|\partial_s u+J\partial_t u=0; \lim_{s\to \infty} u(s,\cdot) =\gamma_x,\lim_{s\to -\infty} u(s,\cdot) =\gamma_y \}.
\end{equation}

Now if $\phi$ is monotone (we will clarify this notion later in Definition \ref{monotone def}), the differential on $CF_*(\Sigma, \phi)$ is defined by

\begin{equation}\label{count}
    \langle \partial x,y \rangle := \#_{\Z/2} (\mathcal{M}_{x,y}/\R).
\end{equation}

Here $\mathcal{M}_{x,y}/\R$ is the natural quotient of $\mathcal{M}_{x,y}$ induced by the $\R$-translation on $\R\times Y_\phi$. Recall that every $J$-holomorphic section $u$ has a \emph{Fredholm index}:
\begin{equation}
    \ind(u)=2c^\tau_1(u) + CZ_\tau(\gamma_x)- CZ_\tau(\gamma_y)
\end{equation}
where $\tau$ is a trivialization of $\ker d\pi$ over the periodic orbits $\gamma_x$ and $\gamma_y$, and $c_1^\tau(u)$ is the relative first Chern number\footnote{The relative first Chern number $c_1^\tau(u)$ is defined as follows: $u^*(\ker d\pi)$ is a symplectic bundle over the domain of $u$. One chooses a generic section $\xi$ of this bundle which on each end is nonvanishing and constant with respect to the trivialization $\tau$. $c_1^\tau(u)$ is defined to be the algebraic count of zeroes of $\xi$.}. For a generic almost complex structure $J$, the set of Fredholm index one $J$-holomorphic sections modulo the natural $\R$ action is a compact $0$-dimensional manifold, and $\#_{\Z/2} (\mathcal{M}^J_{x,y}/\R)$ denotes the mod 2 count of points in the moduli space. If $J$ is generic then $\partial^2=0$, and we will denote by $HF_*(\Sigma, \phi)$ the homology of $(CF_*(\Sigma, \phi),\partial)$. The homology $HF_*(\Sigma,\phi)$ is invariant under symplectic isotopies of $\phi$, see for example \cite{seidel2002symplectic}.

Fixed point Floer homology for a symplectomorphism of a surface has been computed in various cases, see for example \cite{dostoglou1994self,pozniak1994floer,seidel1996symplectic,gautschi2002floer,eftekhary2002floer,cotton2009symplectic,pedrotti2022fixed}.Fixed point Floer homology can also be viewed as a special case of periodic Floer homology, which was calculated for iterations of a Dehn twist in \cite{hutchings2005periodic}.

\subsection{Product and coproduct structures}
Under suitable monotonicity assumptions\footnote{Or with the use of Novikov rings.} (see Definition \ref{def: bundle monotone}), fixed point Floer homology is functorial,  in  the  sense  that  fiberwise  symplectic cobordisms with cylindrical ends induce morphisms between fixed point Floer homologies. In this paper, we focus on (completed) fiberwise symplectic cobordisms coming from the composition of 2 symplectomorphisms. We begin with a review of the concept of symplectic fiber bundles.

\begin{definition}
(\cite{seidel1997floer} Definition 7.1) Let $B$ be a smooth manifold. A symplectic fiber bundle $(E,\pi,\omega)$ over $B$ is a smooth proper submersion $\pi: E\to B$ together with a closed 2-form $\omega\in\Omega^2(E)$ such that the restriction of $\omega$ to any fiber is nondegenerate.
\end{definition}

The mapping torus $Y_\phi$, together with the natural projection $\pi: Y_\phi\to S^1$ and the 2-form $\omega_\phi$, is an example of a symplectic fiber bundle. If $\phi,\psi:\Sigma\to\Sigma$ are two symplectomorphisms, then there is a symplectic fiber bundle $(X,\pi_X, \omega_X)$ over the thrice punctured sphere $B_0$, which, near the three punctures, is symplectomorphic to $[0,\infty)\times Y_\phi$, $[0,\infty)\times Y_\psi$ and $(-\infty, 0]\times Y_{\psi\circ\phi}$ respectively. A more precise description will appear in section \ref{setup}. For now, let us observe that (under assumptions on monotonicity, see Definition \ref{def: bundle monotone}), such a bundle induces a morphism
\begin{equation}
    \bullet: HF_*(\Sigma, \phi)\otimes HF_*(\Sigma, \psi)\longrightarrow HF_*(\Sigma, \psi\circ\phi)
\end{equation}
This is what we call the \emph{product structure} of the fixed-point Floer homology. In particular, if we assume that $\psi$ is isotopic to the identity, then the product structure gives a $H_*(\Sigma;\Z_2)$ module structure on $HF_*(\Sigma,\phi)$. For computations of this module structure, see \cite{seidel1996symplectic,eftekhary2002floer,gautschi2002floer,cotton2009symplectic}. Similarly, one can define, under suitable conditions, the \emph{coproduct structure}:
\begin{equation}
    \Delta: HF_*(\Sigma, \psi\circ\phi)\longrightarrow HF_*(\Sigma, \phi)\otimes HF_*(\Sigma, \psi).
\end{equation}

Like the definition of the differential, the product and coproduct structures have geometric descriptions, this time by counting rigid pseudo holomorphic sections of the bundle $X\to B_0$ with appropriate asymptotes. Namely, if $x$, $y$, $z$ are fixed points of $\phi$, $\psi$ and $\psi\circ\phi$ respectively, and $J$ is a tame almost complex structure (see Section \ref{setup} for the definition) then the moduli space $\mathcal{M}^J_{x,y;z}$ is defined by:
\begin{equation}
    \mathcal{M}^J_{x,y;z} =\left\lbrace u: B_0\to X \;\middle|\;
  \begin{tabular}{@{}l@{}}
    $\pi_X\circ u$=id,\ $u$ is $J$-holomorphic, and\\
    $u$ is asymptotic to $\gamma_x$, $\gamma_y$ and $\gamma_z$ over the\\ three appropriate punctures.
   \end{tabular}
  \right\rbrace
\end{equation}
The product (under suitable monotonicity assumptions, see Definition \ref{def: bundle monotone}) on the chain level is now defined as:
\begin{equation}\label{cobordismcount}
    \langle x\bullet y, z\rangle = \#_{\Z/2}\mathcal{M}^J_{x,y;z}
\end{equation}
where $\#_{\Z/2}\mathcal{M}^J_{x,y;z}$ denotes the mod 2 count of Fredholm index $0$ sections for a generic almost complex structure. The coproduct structure is defined in a similar way.

\subsection{Main results}
In this paper, we calculate the product and coproduct structures of the fixed-point Floer homology of iterations of a single positive Dehn twist. We fix a symplectic surface (possibly with boundary) $(\Sigma,\omega_0)$ and a homologically nontrivial simple closed curve $\gamma\subset\Sigma$. We assume that $\gamma$ does not intersect $\partial \Sigma$ and choose a tubular neighborhood $N$ of $\gamma$ with coordinates $x\in (-\epsilon,1+\epsilon)$ and $y\in S^1=\mathbb{R}/\mathbb{Z}$, where $\omega_0=dx\wedge dy$. A (unperturbed) \emph{positive Dehn twist} along $\gamma$ is a symplectomorphism of $\Sigma$, which has the form
\begin{equation}
\phi_0: (x,y)\mapsto (x,y-x)
\end{equation}
inside $N$, and is the time-1 map of a Morse function $H_0$ outside of $N'=[\epsilon,1-\epsilon]\times S^1\subset N$. To break the degeneracy of the fixed points of $\phi_0$, we perturb $\phi_0$ near each of the $S^1$-family of orbits, breaking the Morse-Bott degeneracy, and denote by $\phi$ the perturbed twist. We further assume, as before, that near each boundary component of $\Sigma$, there are tubular coordinates $x_i\in(-\epsilon_i,0]$, $y_i\in S^1$ and a small real number $\theta_i$ such that $H_0(x_i,y_i)=\theta_i x_i$ (later we will impose some other conditions on $H_0$, see Section \ref{setup}).

Let us assume for now that $\partial \Sigma \ne \emptyset$ or the genus of $\Sigma$ is at least 2. Let $\Sigma_0$ denote $\Sigma-N'$. It was shown, for example in \cite{seidel1996symplectic} or \cite{hutchings2005periodic}, that for all positive integers $m$,
\begin{equation}\label{decomposition}
HF_*(\phi^m)\cong H_*(\Sigma_0;\Z_2)\oplus (\oplus_{i=1}^{m-1} H_*(S^1)).
\end{equation}

The isomorphism can be understood as follows. With our assumption on $\phi$, the chain complex $CF(\phi^m)$ is generated by two types of fixed points: those corresponding to the critical points of $H_0$ on $\Sigma_0$ and those from the breaking of the Morse-Bott $S^1$-family of the fixed points of $\phi_0^m$ inside $N'$. As it turns out, the two types of fixed points generate subcomplexes of $CF(\phi^m)$, whose homologies correspond to the summands of the right-hand side of equation (\ref{decomposition}). Let $m$ and $n$ be two positive integers. With the previously described isomorphism understood, let $\proj$ denote the projection map $HF_*(\phi^m) \to H_*(\Sigma_0;\Z_2)$, let $\cap$ denote the intersection product on $H_*(\Sigma_0;\Z_2)$, and let $\iota$ denote the inclusion map $H_*(\Sigma_0;\Z_2)\to HF_*(\phi^{m+n})$. Our first result is the following:

\begin{theorem}\label{product for multiple twists}
Suppose
\begin{itemize}
    \item If $\gamma$ is non-separating, then $\partial \Sigma \ne \emptyset$ or $\Sigma$ is closed with genus at least $2$;
    \item If $\gamma$ is separating, then each component of $\Sigma-\gamma$ either contains a component of $\partial\Sigma$ or has genus at least 2.
\end{itemize}
Then the product of $HF(\phi^m)$ and $HF(\phi^n)$ is the composition of
\begin{equation}
HF_*(\phi^m)\otimes HF_*(\phi^n)\xrightarrow{\proj\otimes\proj}H_*(\Sigma_0;\Z_2)\otimes H_*(\Sigma_0;\Z_2)\xrightarrow{\quad \cap\quad}H_*(\Sigma_0;\Z_2)\xrightarrow{\quad\iota\quad}HF_*(\phi^{m+n}).
\end{equation}
\end{theorem}

The counterpart for the coproduct structure is more involved. Let us first make the following remark on the decomposition (\ref{decomposition}). The $i$-th component of $\oplus_{i=1}^{m-1} H_*(S^1)$ has an explicit description as follows: let $e^m_i$ (resp. $h^m_i$) be the elliptic (resp. hyperbolic) orbit of $\phi^m$ over the tubular coordinate $x=\frac{i}{m}$ ($i=0,1,\cdots, m$) that arises from perturbing the Morse-Bott degenerate $\phi$ to $\phi_0$ (see section \ref{setup}). Then $e_i^m$ and $h_i^m$ are cycles and the $i$-th component of $\oplus_{i=1}^{m-1} H_*(S^1)$ is spanned by the two homology classes $[e^m_i]$ and $[h^m_i]$. Let us denote by $[e^n_j]$, $[h^n_j]$, $[e^{m+n}_k]$, $[h^{m+n}_k]$ the homology classes appearing in the similar decomposition for $HF_*(\phi^n)$ and $HF_*(\phi^{m+n})$ respectively. Finally, let us recall that for any space $M$, a coproduct structure $\Delta_0$ on $H_*(M;\Z/2)$ is defined as the composition of $\diag_*:H_*(M;\Z/2)\to H_*(M\times M;\Z/2)$ and $H_*(M\times M;\Z/2)\cong H_*(M;\Z/2)\otimes H_*(M;\Z/2)$. Our second main result is the following:
\begin{theorem}\label{coproduct for multiple twists}
Suppose
\begin{itemize}
    \item If $\gamma$ is non-separating, then $\partial \Sigma \ne \emptyset$ or $\Sigma$ is closed with genus at least $2$;
    \item If $\gamma$ is separating, then each component of $\Sigma-\gamma$ either contains a component of $\partial\Sigma$ or has genus at least 2.
\end{itemize}
Then the coproduct $\Delta:HF_*(\phi^{m+n})\to HF_*(\phi^{m})\otimes HF_*(\phi^{n})$ described in the previous section is completely determined by the following:
\begin{enumerate}
    \item When restricted to $H_*(\Sigma_0;\Z_2)\subset HF_*(\phi^{m+n})$, $\Delta$ is equal to
    \begin{equation}\label{eq: coproduct1}
        H_*(\Sigma_0;\Z_2)\xrightarrow{\quad\Delta_0\quad}H_*(\Sigma_0;\Z_2)\otimes H_*(\Sigma_0;\Z_2)\longrightarrow HF_*(\phi^{m})\otimes HF_*(\phi^{n}).
    \end{equation}

    \item For each $[e^{m+n}_{k}]\in \oplus_{i=1}^{m+n-1}H_*(S^1)$, \begin{equation}\label{eq: coproduct2}
        \Delta([e^{m+n}_k])=\sum_{i\in\{0,1,\cdots,m\},\ k-i\in\{0,1,\cdots,n\}} [e^m_{i}]\otimes [e^n_{k-i}].
    \end{equation}
    \item For each $[h^{m+n}_{k}]\in \oplus_{i=1}^{m+n-1}H_*(S^1)$,
    \begin{equation}
        \Delta([h^{m+n}_{k}])=\sum_{i\in\{0,1,\cdots,m\},\ k-i\in\{0,1,\cdots,n\}} [e^m_{i}]\otimes [h^n_{k-i}]+[h^m_{i}]\otimes [e^n_{k-i}].
    \end{equation}
\end{enumerate}
\end{theorem}

\begin{remark}
It was asked in \cite{cotton2009symplectic} how one could get an understanding of the ring structure of $\oplus_{n=0}^\infty HF_*(\Sigma,\phi^n)$ for general $\phi$. Our results compute the algebra and co-algebra structure of $\oplus_{n>0}^\infty HF_*(\Sigma,\phi^n)$ when $\phi$ is the positive Dehn twist. One can use the computation in \cite{seidel1996symplectic} to calculate the case for $n=0$. 

The coproduct structure \footnote{We thank Tim Perutz for pointing out this fact to us.} on $ \oplus_{n\geq 0} HF(\phi^n)$ can be interpreted as the dual of the product structure on the ring $\oplus_{n\leq 0} HF(\phi^n)$ associated with the negative Dehn twist. Likewise, the product structure on $ \oplus_{n\geq 0} HF(\phi^n)$ can be interpreted as the dual of the coproduct structure on $\oplus_{n\leq 0} HF(\phi^n)$.
\end{remark}
\begin{remark}
The proofs in Section \ref{the product} and Section \ref{the coproduct} can be extended to the case where $\phi$ is composed of positive Dehn twists around $C_1\cup C_2\cdots\cup C_n$ where $C_i$ are disjoint and pairwise non parallel, and subject to the condition that each component of $\Sigma \setminus C_1 \cup C_2\cdots\cup C_n$ contains either a component of $\partial \Sigma$ or has genus at least 2. In that case for the product we should replace $\Sigma_0$ with $\Sigma \setminus C_1\cup\cdots\cup C_n$, and for the coproduct, in addition to the above replacement, the formulas (\ref{eq: coproduct1}) and ($\ref{eq: coproduct2}$) of $\Delta$ should be applied separately for each twist region corresponding to $C_i$. 
\end{remark}

\subsection{Strategy of the proof}
In this subsection, we summarize the key ideas behind the proof. As will be explained in more details in Section \ref{setup}, the symplectic fiber bundle $X$ computing the product or the coproduct can be decomposed into two pieces, which we will call $X_D$ and $X_H$. Roughly speaking, $X_D$ is the union of fibers where iterations of the Dehn twist take place, and $X_H$ is (an open neighborhood of) the complement. The first key observation is that, under mild assumptions on the almost complex structure $J$, any $J$-holomorphic section of $X$ that has wrapping number (see Section \ref{section:no_crossing}) $0$ must be completely contained in either $X_D$ or $X_H$. A key technical lemma used in the argument is the ``local energy inequality'' (Lemma \ref{local energy lemma}) that was inspired by Lemma 3.11 of \cite{hutchings2005periodic}, which is reproved in our setting in Section \ref{basic case}. We next observe that for $J$-holomorphic sections, Fredholm index being $0$ implies that the wrapping number is $0$. Thus we only need to focus on sections contained in one of the two pieces.

$J$-holomorphic sections that are contained in $X_H$ are relatively easy to understand: by results of \cite{PSS,frauenfelder2007hamiltonian,lanzat2016hamiltonian} (also known as the PSS isomorphism), these sections contribute to the intersection product or the coproduct of $H_*(\Sigma_0;\Z_2)$. 

Sections that are contained in $X_D$ are more interesting. When computing the product structure, we are again able to rule out most of them using the local energy inequality. For the remaining sections, we use a translation trick (equation (\ref{eq: translation trick})) together with the PSS isomorphism to conclude that they contribute zero. For the coproduct structure, sections contained in $X_D$ do make contributions. To understand the contributions of these sections, we give a concrete description of the (unperturbed) moduli spaces, see Proposition \ref{morse-bott moduli space}, whose proof involves a deformation argument and a concrete construction. Finally, a Morse-Bott correspondence theorem (explained in Section \ref{subsection: MB}) finishes the calculation.

\begin{remark}
For the computation of the coproduct structure, we expect that sections contained in the twist region $X_D$ could also be understood by results of \cite{abbondandolo2010floer,cieliebak2009role}. More precisely, we think that sections contained in $X_D$ could be viewed as part of curves counted in the product of the Floer homology of $T^*S^1$. By the results in
\cite{abbondandolo2010floer,cieliebak2009role}, the Chas-Sullivan loop product on the singular homology of the loop space of $S^1$ is expected to give us directly the coefficients in the coproduct structure. However, this method is not used in the current paper. Instead, we explicitly determine the Morse-Bott moduli spaces, with the expectation that such constructions will allow us to calculate similar cobordism maps for the periodic Floer homology in the future.
\end{remark}

\subsection{Directions for future work}
\textbf{General pairs of symplectomorphisms.} While in this paper we computed the product and coproduct structures for iterations of a single Dehn twist, it is interesting to investigate the same question for $\phi$, $\psi$ Dehn twists along different circles. In general we expect this to be hard, probably requiring techniques beyond this paper.

For a different direction, one could also investigate the situation where $\phi$, $\psi$ are of finite type. The fixed point Floer homology for finite type maps was computed in \cite{gautschi2002floer}. It might be possible that the separating results in the current paper still holds under certain conditions, yielding a computation for the cobordism maps.

\textbf{Periodic Floer homology.} Fixed point Floer homology can be viewed as the degree $d=1$ case of the periodic Floer homology (PFH), which is a more complicated Floer theory one can associate to a symplectomorphism of a surface. While the general construction of cobordism maps in PFH relies on Seiberg-Witten theory, special cases in which cobordism maps can be defined via counts of holomorphic curves or buildings have been worked out. For works in this direction, see for example \cite{chen2017cobordism} and \cite{rooney2018cobordism}. To generalize the results in the current paper, one could try to calculate the product and coproduct structures of PFH of a single Dehn twist over a surface. While getting a complete answer could be difficult, one could start by understanding the (multi-)sections contained in the twist region. We hope that the detailed analysis presented in the current paper could give a hint on what the moduli spaces should look like.

\subsection{Organization of the paper}
The rest of the paper is organized as follows. In Section \ref{setup} we review the relevant geometric setup in detail. In particular, we give an explicit description of the symplectic fiber bundles $X_{m,n}$ and $X^{m,n}$ that are used to define the cobordism maps. In Section \ref{basic case}, we follow an idea of \cite{hutchings2005periodic} to establish ``no crossing'' results for a special almost complex structure, Lemma \ref{no crossing lemma} and Lemma \ref{coproduct no crossing lemma}. In Section \ref{general no crossing}, we use an SFT compactness argument to prove Theorem \ref{classification for general J}, which is a generalization of the ``no crossing'' results in Section \ref{basic case} for a general almost complex structure $J$. Finally, in Section \ref{the product} and Section \ref{the coproduct}, we prove our main results, Theorem \ref{product for multiple twists} and Theorem \ref{coproduct for multiple twists}.

\subsection{Acknowledgements}
First and foremost we would like to thank our advisor Michael Hutchings for his consistent support and guidance. We would like to also thank Kiran Luecke and Yixuan Li for clarifications on the loop product. We would also like to thank Jo Nelson and Tim Perutz for helpful comments.

Yuan Yao would like to acknowledge the support of the Natural Sciences and Engineering Research Council of Canada (NSERC), PGSD3-532405-2019.
Cette recherche a été financée par le Conseil de recherches en sciences naturelles et en génie du Canada (CRSNG), PGSD3-532405-2019. Ziwen Zhao was partially supported by NSF grant DMS-2005437.

\section{The setup}\label{setup}
We fix a symplectic surface, possibly with boundary, $(\Sigma,\omega_0)$. Let $\phi$ be any symplectomorphism. 
Let $Y_\phi$ be the mapping torus of $\phi$, let $\pi$ be the projection $Y_\phi\to S^1$, and let $\omega_\phi$ be the induced closed two-form. The mapping torus has the structure of a stable Hamiltonian manifold, where the 1-form is $dt$, the 2-form is $\omega_\phi$ and the associated Reeb vector field is $R=\partial_t$. Closed integral curves of the Reeb vector field are called \emph{Reeb orbits}, and they are called \emph{nondegenerate} if the linearized first return map doesn't have 1 as an eigenvalue. We call a nondegenerate Reeb orbit \emph{hyperbolic} if the eigenvalues are real, and \emph{elliptic} otherwise. In this paper, we will be mainly interested in Reeb orbits of degree 1, that is, those who cover once under the projection map $\pi$.

We fix a homologically nontrivial simple closed curve $\gamma\subset\Sigma$. As mentioned in the introduction, we choose a tubular neighborhood $N$ of $\gamma$ with coordinates $x\in (-\epsilon,1+\epsilon)$ and $y\in S^1=\mathbb{R}/\mathbb{Z}$, and $\omega_0=dx\wedge dy$. The \textbf{ (unperturbed) positive Dehn twist} along $\gamma$ is a symplectomorphism of $\Sigma$, which has the form
\begin{equation}
\phi_0: (x,y)\mapsto (x,y-x)
\end{equation}
inside $N$, and is the time-1 map of a Hamiltonian $H_0$ outside of $N'=[\epsilon,1-\epsilon]\times S^1\subset N$.
We require $H_0$ takes the following form on $N \setminus N'$:
\begin{enumerate}
\item $H_0(x,y)=\frac{1}{2}x^2$ in $(-\epsilon,\epsilon)\times S^1\subset N$, and
\item $H_0(x,y)=\frac{1}{2}(x-1)^2$ in $(1-\epsilon,1+\epsilon)\times S^1\subset N$.
\end{enumerate}
We assume on $\Sigma \setminus N$, the function $H_0$ is a $C^2$ small Morse function, so that the associated time-1 map is non-degenerate.
We further assume that near each boundary component of $\Sigma$, there are tubular coordinates $x_i\in(-\epsilon_i,0]$, $y_i\in S^1$ and a small real number $\theta_i$ such that $H_0(x_i,y_i)=\theta_i x_i$. 

We note the unperturbed positive Dehn twist is non-degenerate, except for the Morse-Bott $S^1$ family of periodic orbits corresponding to $x=0$ and $x=1$. 
We shall later consider iterations of $\phi_0$, which we denote by $\phi_0^n$. By the above, on $N$, the map $\phi_0^n$ takes
\[
(x,y) \mapsto (x,y-nx)
\]
on $N$ and looks like $nH_0$ outside of $N'$. We assume both $nH_0$ (in $C^2$ norm) and $n\theta_i$ are small.

We note that in order to define the fixed-point symplectic homology, we need the symplectomorphisms to be nondegenerate (equivalently, that the Reeb orbits are cut out transversely). Since the symplectomorphism $\phi_0^n$ on $\Sigma-N$ is the time-1 map of a Hamiltonian $n H_0$, this is achieved outside of $N$ by requiring that $H_0$ be a $C^2$-small Morse function. Inside the tubular region $N$, Reeb orbits come in Morse-Bott $S^1$ families. Following \cite{hutchings2005periodic}, We overcome this technical difficulty by perturbing $\phi_0$ (in a small neighborhood of finitely many values of $x$ over which Reeb orbits exist) in a Hamiltonian way, which amounts to adding a Hamiltonian perturbation term. For example, near $x=0$, we can modify $H_0$ to be $(\frac{1}{2}x^2+\lambda(x) h(y)))$, where $\lambda(x)$ is a cutoff function supported in $(-\delta,\delta)_x$ with $\lambda(0)=1$ as a non-degenerate local max, and $h:S^1_y\to \mathbb{R}$ is a small perfect Morse function. We perform this kind of perturbation for each $S^1$ family of fixed points in $N$. We always assume that the Hamiltonian perturbation only takes place in the union of all intervals $(x_i-\delta,x_i+\delta)$ (where $x_i$'s are the $x$-coordinates for all possible Morse-Bott $S^1$-families) for some positive real number $\delta$ much smaller than $\epsilon$ (we'll later call the complement of these intervals the \emph{unperturbed range}). Once this is done, viewed from the perspectice of the mapping torus $Y_{\phi_0^n}$, the $S^1$ family of fixed points at $x=\frac{i}{n}$ become perturbed to a pair of Reeb orbits (one elliptic and one hyperbolic).

With the above \textbf{perturbed positive Dehn twist}, which we denote by $\phi^n$, we can define its fixed point Floer homology $HF(\Sigma,\phi^n)$ after we pick a generic $\phi^n$ compatible almost complex structure $J$ on $Y_{\phi^n}\times \mathbb{R}$.  We next describe the symplectic fiber bundle  that allows us to define product and coproduct structures on $HF(\Sigma,\phi^n)$. We first describe the construction for the unperturbed positive Dehn twists, then perturb to break the Morse-Bott degeneracy. The reason we describe the Morse-Bott situation in detail is because for the coproduct computation we will be enumerating $J$-holomorphic sections in the Morse-Bott setting, then we will use Morse-Bott theory to convert that to enumerations of $J$-holomorphic sections in the nondegenerate setting.

Recall that, given two symplectomorphisms, there is a symplectic fiber bundle $(X,\pi_X, \omega_X)$ over the thrice punctured sphere $B_0$, which is modelled by the symplectizations of mapping tori over the punctures. We now describe in more details what the bundle $X_{m,n}$ used in computing the product structure
\begin{equation}
  \bullet: HF_*(\phi^m)\otimes HF_*(\phi^n)\longrightarrow HF_*(\phi^{m+n})
\end{equation}
looks like. The description for the bundle $X^{m,n}$ used to compute the coproduct structure is almost identical, and we will mention at the end of this section what changes need to be made.

We designate two of the punctures of $B_0$ as ``positive'', and the other as ``negative''. Choose local conformal coordinates $s_i\in[0,\infty)$ and $t_i\in S^1 (i=1,2)$ near the 2 positive punctures of $B_0$, and local conformal coordinates $s_{-\infty}\in(-\infty,0]$, $t_{-\infty}\in S^1$ near the negative puncture. Fix also a smooth function $g_{m,n}:B_0\to S^1$ such that $dg_{m,n}=md t_1$ near the first positive puncture, $dg_{m,n}=nd t_2$ near the second positive puncture and $dg_{m,n}=(m+n)dt_{-\infty}$ near the negative puncture. Define the closed one-form $\beta_{m,n}=dg_{m,n}$.

Let $\Sigma_0=\Sigma-N'$. We now describe the fiberwise symplectic cobordism $X_{m,n}$ as the union of two fiberwise symplectic cobordisms $X_D$ and $X_H$ as follows. Topologically, $X_D=B_0\times N$ and $X_H=B_0\times \Sigma_0$. In order to describe $\phi_0$ as the time-1 map of the Hamiltonian $H_0$ near the two ends of the tubular region $N$ we choose coordinates $(p,x_L,y_L)\in B_0\times(-\epsilon,\epsilon)\times S^1$ and $(p,x_R,y_R)\in B_0\times(1-\epsilon,1+\epsilon)\times S^1$ for the two ends of $B_0\times (N-N')\subset X_H$ and impose that
\begin{enumerate}
\item $\omega_0=dx_L\wedge dy_L$ or $dx_R\wedge dy_R$ in the two components of $N-N'$,
\item $H_0(x_L,y_L)=\frac{1}{2}x_L^2$ in $(-\epsilon,\epsilon)\times S^1\subset N$, and
\item $H_0(x_R,y_R)=\frac{1}{2}(x_R-1)^2$ in $(1-\epsilon,1+\epsilon)\times S^1\subset N$.
\end{enumerate}

Topologically, the 4-manifold $X$ is defined to be $X=X_H\bigcup X_D/\sim$, where we identify points $(p,x,y)\in B_0\times(-\epsilon,\epsilon)_x\times S^1_y\subset X_D$ with $(p,x,y)\in B_0\times(-\epsilon,\epsilon)_{x_L}\times S^1_{y_L}\subset X_H$, and $(p,x,y)\in B_0\times(1-\epsilon,1+\epsilon)_x\times S^1_y\subset X_D$ with $(p,x,y+g_{m,n}(p))\in B_0\times(1-\epsilon,1+\epsilon)_{x_R}\times S^1_{y_R}\subset X_H$.

Now we define the fiberwise symplectic 2-form $\omega_{X,0}$ to be $dx\wedge dy+d(\frac{1}{2}x^2\beta_{m,n})$ in $X_D$, and $\omega_0+d(H_0\beta_{m,n})$ in $X_H$. It is easy to see that the two definitions agree in $X_D\supset B_0\times (-\epsilon,\epsilon)_x\times S^1_y=B_0\times (-\epsilon,\epsilon)_{x_L}\times S^1_{y_L}\subset X_H$. To see that the two definitions agree in $X_D\supset B_0\times (1-\epsilon,1+\epsilon)_x\times S^1_y=B_0\times (1-\epsilon,1+\epsilon)_{x_R}\times S^1_{y_R}\subset X_H$, we calculate
\begin{align}
\begin{split}
dx_R\wedge dy_R+ d(\frac{1}{2}(x_R-1)^2\beta_{m,n}) &= dx\wedge(dy+\beta_{m,n})+d(\frac{1}{2}(x-1)^2\beta_{m,n})\\
 &= dx\wedge dy+dx\wedge\beta_{m,n}+(x-1)dx\wedge\beta_{m,n}\\&= dx\wedge dy+d(\frac{1}{2}x^2\beta_{m,n}).
 \end{split}
\end{align}
We remark that over the positive punctures, the symplectic fiber bundle defined above are isomorphic to $[0,\infty)$ times the mapping tori $Y_{\phi_0^m}$, $Y_{\phi_0^n}$, and over the negative puncture, the above fiber bundle is modelled by $(-\infty, 0]$ times the mapping torus $Y_{\phi_0^{m+n}}$.

The above fiber bundle have Morse-Bott degeneracies in its Reeb orbits at each of its punctures. To arrive at the definition of $X_{m,n}$, we perturb the Reeb orbits to be non-degenerate as before. Since we are working in the language of a symplectic fiber bundle, we achieve this by adding a Hamiltonian perturbation term to $\omega_{X,0}=dx\wedge dy+d(\frac{1}{2}x^2\beta_{m,n})$.

As before, near $x=0$, we can modify $\omega_{X,0}$ to be $\omega_X=dx\wedge dy+d((\frac{1}{2}x^2+\lambda(x) h(y))\beta_{m,n})$, where $\lambda(x)$ is a cutoff function supported in $(-\delta,\delta)_x$ with $\lambda(0)=1$ as a nondegenerate local max, and $h:S^1_y\to \mathbb{R}$ is a small perfect Morse function. We assume as before that the Hamiltonian perturbation only takes place in the union of all intervals $(x_i-\delta,x_i+\delta)$ (where $x_i$'s are the $x$-coordinates for all possible Morse-Bott $S^1$-families) for some positive real number $\delta$ much smaller than $\epsilon$ (we'll later call the complement of these intervals the \emph{unperturbed range}). Once this is done, near the first (resp. the second) positive puncture, $x=\frac{i}{m}$ (resp. $x=\frac{j}{n}$) each correspond to a pair of Reeb orbits (one elliptic and one hyperbolic), and near the negative puncture $x=\frac{k}{m+n}$ each correspond to a pair of Reeb orbits (one elliptic and one hyperbolic).

This defines the fiberwise symplectic 2-form $\omega_X$, we now describe the symplectic structure on $X^{m,n}$.  It's illustrated in e.g. \cite{chen2017cobordism} that from the fiberwise symplectic cobordism $(X,\pi_X,\omega_X)$ one can construct a symplectic form $\Omega_X=\omega_X+K \pi_X^*\omega_B$ where $K$ is a large positive number, and $\omega_B$ is an area form on $B_0$. Without loss of generality, we assume from now on that $K \pi_X^*\omega_B=ds_i\wedge dt_i$ near the punctures of $B_0$. 

This concludes the definition of $X_{m,n}$, we now equip it with a tame almost complex structure $J$.

\begin{definition}
An almost complex structure on $(X_{m,n},\omega_X, \pi)$ is called tame if the following conditions are satisfied
\begin{enumerate}
    \item Near the punctures of $B_0$ where the symplectic fiber bundle is isomorphic to $Y_{\phi^n}\times [0,\infty)$ (resp. $Y_{\phi^m}\times [0,\infty)$ or $Y_{\phi^{m+n}}\times (-\infty,0]$), the almost complex structure is given by the restriction of a $\phi^n$ (resp. $\phi^m$, or $\phi^{m+n}$) compatible almost complex structure.
    \item Away from the cylindrical neighborhoods around the punctures of $B_0$, the almost complex structure $J$ is tamed by the symplectic form $\Omega_X$.
\end{enumerate}
\end{definition}

Then for generic tame $J$, if $x$, $y$, $z$ are fixed points of $\phi^n$, $\phi^m$ and $\phi^{m+n}$ respectively, the moduli space  $\mathcal{M}^J_{x,y;z}$, defined by 
\begin{equation}
    \mathcal{M}^J_{x,y;z} =\left\lbrace u: B_0\to X \;\middle|\;
  \begin{tabular}{@{}l@{}}
    $\pi_X\circ u$=id,\ $u$ is $J$- holomorphic, and\\
    $u$ is asymptotic to $\gamma_x$, $\gamma_y$ and $\gamma_z$ over the\\ three appropriate punctures.
   \end{tabular}
  \right\rbrace
\end{equation}
is a manifold whose dimension is given by the Fredholm index formula
\begin{equation}
\ind(u) = 1 + 2 \langle c_1^\tau(TX_{m,n}),[u]\rangle+  \CZ_\tau(\gamma_x)+\CZ_\tau(\gamma_y) - \CZ_\tau(\gamma_z).
\end{equation}
Here $\tau$ denotes a choice of fixed trivializations around each Reeb orbit, and $CZ_\tau$ denotes the Conley-Zehnder indices of Reeb orbits with respect to this trivialization. Similarly the relative first Chern class $c^\tau_1$ is also determined by this choice of trivialization. See Section \ref{the product} for our specific choices of $\tau$.
The product on the chain level is now defined as:
\begin{equation}\label{cobordismcount}
    \langle x\bullet y, z\rangle = \#_{\Z/2}\mathcal{M}^J_{x,y;z}
\end{equation}
where $\#_{\Z/2}\mathcal{M}^J_{x,y;z}$ denotes the mod 2 count of Fredholm index $0$ sections. See Section \ref{the product} for the details of this computation.

For the coproduct structure, the symplectic fiber bundle $X^{m,n}$ is defined almost verbatim, so we only highlight the minor changes that need to be made. We again begin with the thrice punctured sphere $B_0$, but this time choose one of the three punctures as the ``positive puncutre'' with a local conformal coordinate $(s_{\infty}, t_{\infty})\in [0,\infty)\times S^1$, and choose the other two punctures as the two ``negative'' punctures with local coordinates $(s_i,t_i)\in (-\infty, 0]\times S^1$. Fix a smooth function $g^{m,n}:B_0\to S^1$ such that $dg^{m,n}=(m+n)dt_{\infty}$ near the positive puncture and $dg^{m,n}=mdt_1$ and $ndt_2$ near the two negative punctures respectively. We again define $X^{m,n}$ by gluing two trivial fiber bundles $X_H=B_0\times \Sigma_0$ and $X_D=B_0\times N$, but this time the gluing map for the right side of $X_D$ is
\begin{equation}
B_0\times(1-\epsilon,1+\epsilon)_x\times S^1_y\ni (p,x,y)\sim (p,x,y+g^{m,n}(p))\in B_0\times(1-\epsilon,1+\epsilon)_{x_R}\times S^1_{y_R}
\end{equation}

Let $\beta^{m,n}=dg^{m,n}$. We similarly define the (unperturbed) fiberwise symplectic 2-form $\omega_{X,0}$ to be $dx\wedge dy+d(\frac{1}{2}x^2\beta_{m,n})$ in $X_D$ and $\omega_0+d(H_0\beta_{m,n})$ in $X_H$. As before, we perturb $\omega_{X,0}$ to be $\omega_X$ in order to break the Morse-Bott degeneracy, and we always assume that such perturbation is supported in a $\delta$ neighborhood of Reeb orbits inside $X_D$.

We also define the symplectic form $\Omega_X$ on $X^{m,n}$, and the notion of \emph{tame} almost complex structures with respect to $\Omega_X$. Then the coproduct is defined by considering the moduli space of $J$-holomorphic sections $\mathcal{M}^J_{z;x,y}$ where $z$ is a fixed point of $\phi^{m+n}$ we think of as the input, and $x$ and $y$ are fixed points of $\phi^n$ and $\phi^m$ we think of as outputs. For generic $J$ this moduli space is a manifold and the coproduct is defined by the mod 2 count of Fredholm index 0 $J$-holomorphic sections.

\section{The ``no crossing'' results for unperturbed \texorpdfstring{$J$}{J}} \label{section:no_crossing} \label{basic case}
Throughout this section, the symplectic fiber bundle $X$ refers to either $X_{m,n}$ or $X^{m,n}$.
As explained in the introduction, to define the cobordism map, we count $J$-holomorphic sections that asymptote to appropriate Reeb orbits. We'll prove some key properties about such $J$-holomorphic sections for some particularly nice almost complex structures. Before doing that, let us introduce some terminologies. Following \cite{mcduff2012j}, the \emph{vertical distribution} $\text{Ver}$ is the kernel of $d\pi_X:TX\to TB_0$. The \emph{horizontal distribution} $\text{Hor}$ is defined as $\text{Hor}_x:=\{u\in T_xX\ |\ \omega_X(u, v)=0 \ \forall v\in \text{Ver}_x\}$.

\begin{definition} \label{def:fibration_compatible_J}
(\cite{mcduff2012j} Definition 8.2.6) An almost complex structure on $(X,\omega_X)$ is called fibration-compatible if the following holds:
\begin{enumerate}
    \item The projection $\pi_X$ is holomorphic: $J\circ d\pi = d\pi\circ j_0$.
    \item For every $p\in B_0$, the restriction $J_p$ of $J$ to $\pi_X^{-1}(p)$ is tamed by $\omega_X|_z$.
    \item The horizontal distribution $\textup{Hor}$ is preserved by $J$.
\end{enumerate}
\end{definition}

Note that by definition, there is a one-to-one correspondence between fibration-compatible almost complex structures and $\omega_X$-tame almost complex structures on the vertical distribution.

In this section we only consider fibration-compatible almost complex structures. For a fibration-compatible $J$, all the horizontal sections are $J$-holomorphic (a section $u:B_0\to X$ is horizontal, if $du(TB_0)\subset \text{Hor}$). It's also clear from the definition that for any $J$-holomorphic section $u:B_0\to X$ and $v\in TB_0$,
\begin{equation}\label{local ineq}
u^*\omega_X(v, j_0 v)\ge 0
\end{equation}
and the equality holds if and only if $du(v)\in \text{Hor}$.

Following \cite{hutchings2005periodic} Lemma 3.11, we now establish a local energy inequality for $J$-holomorphic sections. To state the inequality, for any $x\in (-\epsilon, 1+\epsilon)$ we let $F_x$ denote the 3-manifold $B_0\times \{x\}\times S^1_y\subset X_D $. Likewise, let $F_{[x_1,x_2]}$ denote the 4- manifold $B_0\times [x_1,x_2]_x\times S^1_y\subset X_D $. The first homology group of $X_D = B_0\times(-\epsilon, 1+\epsilon)_x\times S^1_y$ is $\mathbb{Z}^3$, generated by $[S^1_{t_1}]$, $[S^1_{t_2}]$ and $[S^1_{y}]$. We identify $p[S^1_y]+q_1[S^1_{t_1}]+q_2[S^1_{t_2}]\in H_1(X_D)$ with a tuple $(p,q_1,q_2)\in \mathbb{Z}^3$.

\begin{lemma}\label{local energy lemma}
(Local energy inequality) Let $C$ be a $J$-holomorphic section $u:B_0\to X$ which is not horizontal. Assume that $C$ intersects $F_x$ transversely and that $C\cap F_x\ne \emptyset$ for some $x$ in the unperturbed range. Orient each circle in $C\cap F_x$ using the boundary orientation of $C\cap F_{[x-\epsilon',x]}$ (for a small $\epsilon'$) induced by $j_0$. Under this orientation, let  $(p,q_1,q_2)$ denote the homology class of $C\cap F_x$, then we have 
\begin{equation}\label{homology inequality}
    p+x(mq_1+nq_2)>0
\end{equation}
\end{lemma}

\begin{proof}[Proof of Lemma \ref{local energy lemma}]
Choose $x_1<x<x_2$ such that
\begin{enumerate}
\item $C$ intersects both $F_{x_1}$ and $F_{x_2}$ transversely,
\item $x-x_1=x_2-x$, and
\item $[x_1,x_2]$ is contained in the unperturbed range.
\end{enumerate}

We orient $C\cap F_{x_1}$ and $C\cap F_{x_1}$ in the same way as stated in the lemma, so we have $\partial(C\cap F_{[x_1,x_2]})$ = $C\cap F_{x_2}-C\cap F_{x_1}$. Also notice that with the specified orientations, $C\cap F_{x_1}$, $C\cap F_{x_2}$ and $C\cap F_{x}$ all have the same homology class in $H_1(X_D)$. Now we note that $C\cap F_{[x_1,x_2]}$ is not horizontal, otherwise $C$ has to be horizontal everywhere by unique continuation (recall that all horizontal sections are $J$-holomorphic).

Using the inequality (\ref{local ineq}), we have (in the following the one-form $\beta$ refers to either $\beta_{m,n}$ or $\beta^{m,n}$)
\begin{align}
\begin{split}
0 &< \int_{C\cap F_{[x_1,x_2]}} u^*\omega_X = \int_{C\cap F_{[x_1,x_2]}} u^*(dx\wedge dy+d(\frac{1}{2}x^2\beta))\\
  &= \int_{\partial(C\cap F_{[x_1,x_2]})} u^*(x dy+\frac{1}{2}x^2\beta)\\
  &= x_2\int_{C\cap F_{x_2}} u^*dy + \frac{1}{2}x_2^2\int_{C\cap F_{x_2}} u^*\beta- x_1\int_{C\cap F_{x_1}} u^*dy - \frac{1}{2}x_1^2\int_{C\cap F_{x_1}} u^*\beta\\
  &=x_2 p+\frac{1}{2}x_2^2 (mq_1+nq_2)-x_1 p-\frac{1}{2}x_1^2(mq_1+nq_2)\\
  &=(x_2-x_1)(p+x(mq_1+nq_2)).
\end{split}
\end{align}
\end{proof}
\begin{remark}
  It is clear from the proof that if $u$ is a horizontal section, then the equality $p+x(mq_1+nq_2)=0$ holds.
\end{remark}
Let's assume for the moment that $\Sigma_0=\Sigma-N'$ is connected. Following \cite{hutchings2005periodic}, we define the wrapping number of $J$-holomorphic sections with cylindrical asymptotes to be $\eta(C)=\# C\bigcap (B_0\times\{P_0\})$, where $P_0\in\Sigma_0$ is not a critical point of $H_0$. The algebraic intersection number does not depend on the choice of $P_0$.
We also note that $\eta(C)$ is identically zero if $\partial\Sigma\ne\emptyset$ and under the additional assumption that near each boundary component, $J$ is induced from the vertical almost complex structure that sends $\partial_{x_i}$ to $\partial_{y_i}$. The reason is that once such a $J$ is chosen, no $J$-holomorphic sections can enter the boundary region by the following maximum principle, so one can choose $P_0$ inside one of the boundary region and easily see that $C\bigcap (B_0\times\{P_0\})=\emptyset$.

\begin{lemma}\label{mp}
(Maximum principle) Let $J$ be a fibration-compatible almost complex structure on $X$ that sends $\partial_{x_i}$ to $\partial_{y_i}$ near each boundary component of $\Sigma$. Here the index $i$ labels the different boundary components of $\Sigma$.
Let $V$ denote an open subset of $B_0$ with local coordinates $(s,t)$. Let $\Tilde{u}:V\to X_H$ denote a $J$-holomorphic section, which in coordinates look like $(s,t)\mapsto (s,t,x_i(s,t),y_i(s,t))$. We further assume for $(s,t)\in V$, the pair $(x_i(s,t),y_i(s,t))$ is in a neighborhood of the $i$th boundary component of $\Sigma$. Then $x_i(s,t)$ is a harmonic function.
\end{lemma}
\begin{proof}
The setup is almost identical to that of Lemma \ref{dictionary to 2d}, except that the Hamiltonian function is $\theta_i x_i$ instead of $x_i^2/2$. In particular, the horizontal lifts are
\begin{equation}
    \partial_s^\#=\partial_s-\theta_i F\partial_y,\ \partial_t^\#=\partial_t-\theta_i G\partial_y
\end{equation}
and we have a similar equation:
\begin{equation}
\begin{cases}
\frac{\partial x_i}{\partial t}+\frac{\partial y_i}{\partial s}+\theta_iF=0\\
\frac{\partial y_i}{\partial t}-\frac{\partial x_i}{\partial s}+\theta_iG=0.
\end{cases}
\end{equation}
Notice that $\beta = F(s,t)ds+G(s,t)dt$ being closed tells us that $\frac{\partial F}{\partial t} = \frac{\partial G}{\partial s}$, so the conclusion follows by a simple calculation.
\end{proof}
In particular the above lemma implies that as long as $J$ is chosen in a neighborhood of each of the boundary components of $\Sigma$ to be fibration compatible and sends $\partial_{x_i}$ to $\partial_{y_i}$, no $J$-holomorphic section may approach the boundary components of $\Sigma$.
\begin{remark}
  If $\Sigma_0=\Sigma-N'$ is not connected, i.e. $\gamma$ is separating, we can define two wrapping numbers $\eta_1$ and $\eta_2$ for each of the connected components of $\Sigma_0$. It is clear from the above arguments that if each connected components of $\Sigma_0$ contains part of $\partial \Sigma$, then all the wrapping numbers vanish automatically.
\end{remark}
\begin{remark} \label{remark_nonnegative_wrapping_number}
The wrapping numbers of any $J$-holomorphic section are non-negative, see \cite{hutchings2005periodic} Lemma 4.3 for a proof (the proof generalizes in our setting as well, see the proof of Lemma \ref{no crossing lemma}).
\end{remark}

The first main result of this section is the following ``no crossing'' lemma:
\begin{lemma}\label{no crossing lemma}
Assume $J$ is a fibration-compatible almost complex structure on $X_{m,n}$. If $C$ is a $J$-holomorphic section of $X_{m,n}$ such that all wrapping numbers are zero, then $C$ is either contained in $X_H$ or contained in $X_D$.
\end{lemma}

\begin{proof}[Proof of Lemma \ref{no crossing lemma}]
We first consider the case where $C$ is not horizontal. Suppose there is a nonhorizontal $J$-holomorphic section $C$ that is neither contained in $X_H$ nor in $X_D$, then since $C$ is connected, we can either find some $\epsilon_1\in (\delta,\epsilon)$ such that $C$ intersects both $F_{\epsilon_1}$ and $F_{-\epsilon_1}$ transversely and $C\bigcap F_{\pm\epsilon_1}\ne\emptyset$, or some $\epsilon_1\in (\delta,\epsilon)$ such that $C$ intersects both $F_{1+\epsilon_1}$ and $F_{1-\epsilon_1}$ transversely and $C\bigcap F_{1\pm\epsilon_1}\ne\emptyset$. Without loss of generality, let us assume the first situation happens. (If the second situation happens, the following proof works almost verbatim; the only change one needs to make is that, if we use $(p^{\pm},q_1^{\pm},q_2^{\pm})$ to denote the homology classes of $C\bigcap F_{1\pm\epsilon_1}$, then the condition $\eta=0$ translates to $p^{\pm}+m q_1^{\pm}+n q_2^{\pm}=0$.)

Let $(p^{\pm},q_1^{\pm},q_2^{\pm})$ denote the homology classes of $C\bigcap F_{\pm\epsilon_1}$. Since all the wrapping numbers vanish, we observe that $p^{\pm}=0$. To see this fact, notice that we can choose $P_0$ to be $(\pm\epsilon_1, y_0)\in\Sigma_0$ for some fixed $y_0\in S^1$, then $\# C\bigcap (B_0\times\{P_0\})$ is precisely the number of times $C\bigcap F_{\pm\epsilon_1}$ passes through $y_0$, which equals $p^{\pm}$.

Lemma \ref{local energy lemma} tells us that
\begin{equation}
\epsilon_1(mq_1^{+}+nq_2^{+})>0>\epsilon_1(mq_1^{-}+nq_2^{-})
\end{equation}
which implies that
\begin{equation}
mq_1^{+}+nq_2^{+}\ge 1,\ mq_1^{-}+nq_2^{-}\le -1.
\end{equation}
Now we consider $C_{[-\epsilon_1, \epsilon_1]}:=C\bigcap F_{[-\epsilon_1, \epsilon_1]}$ which is a surface with boundary $C\bigcap F_{\epsilon_1}-C\bigcap F_{-\epsilon_1}$, possibly with positive and negative punctures at $x=0$. Let $d_1$(resp. $d_2$, $d_{-\infty}$) $\in\{0,1\}$ denote the number of punctures of $C_{[-\epsilon_1,\epsilon_1]}$ that project to the first positive puncture (resp. the second positive puncture, the negative puncture) of $B_0$. Notice that the two Reeb orbits over $x=0$ at the first positive puncture (resp. the second puncture, the negative puncture) have the homology class $(0,1,0)$ (resp. $(0,0,1)$ and $(0,1,1)$), so we have
\begin{equation}
d_1(0,1,0)+d_2(0,0,1)+(0,q_1^+,q_2^+)=d_{-\infty}(0,1,1)+(0,q_1^-,q_2^-)
\end{equation}
and hence
\begin{equation}
2\le (mq_1^+ +nq_2^+)-(mq_1^- + nq_2^-)=(m+n)d_{-\infty}-md_1-nd_2
\end{equation}
which implies that
\begin{equation}
d_{-\infty}=1.
\end{equation}

The above equation implies that $C$ has no other negative punctures. So for any $\epsilon_2\in (\delta,\epsilon)$, the section $C$ cannot intersect both $F_{1-\epsilon_2}$ and $F_{1+\epsilon_2}$, because otherwise, the same argument as above would imply that $C$ has another negative puncture asymptotic to one of the Reeb orbits over $x=1$, a contradiction. So there are two remaining possibilities:
\begin{enumerate}
    \item $C\bigcap F_{1-\epsilon_2}=\emptyset$. Let us consider $C\bigcap F_{[\epsilon_1,1-\epsilon_2]}$. For this part of $C$, there are no negative punctures or positive punctures, so we conclude that $\partial(C\bigcap F_{[\epsilon_1,1-\epsilon_2]}) = -C\bigcap F_{\epsilon_1}$ is null homologous in $H_1(X_D)$, which contradicts the fact that $mq_1^++nq_2^+\ge1$.

    \item $C\bigcap F_{1+\epsilon_2}=\emptyset$. Let us consider $C\bigcap F_{[\epsilon_1,1+\epsilon_2]}$. For this part of $C$, let $a$ (resp. $ b$)$\in\{0,1\}$ denote the number of punctures $C$ has at $x=1$ that project to the first (resp. second) positive puncture of $B_0$. Observe that the Reeb orbits at $x=1$ near the first (resp. second) positive puncture have the homology class $(-1,1,0)$ (resp. $(-1,0,1)$) $\in H_1(X_D)$, so we conclude that
    \begin{equation}
    a(-m,1,0)+b(-n,0,1)=(0,q_1^+,q_2^+)
    \end{equation}
    which in turn implies that $a=b=0$. But then it follows that $q_1^+=q_2^+=0$, contradicting $mq_1^{-}+nq_2^{-}\le -1$.
\end{enumerate}

Finally, we consider the case where $C$ is horizontal. Suppose there exist such horizontal section $C$ that is neither contained in $X_D$ nor in $X_H$, then again WLOG we can assume that there exist some $\epsilon_1\in (\delta,\epsilon)$ such that $C$ intersects both $F_{\epsilon_1}$ and $F_{-\epsilon_1}$ transversely and $C\bigcap F_{\pm\epsilon_1}\ne\emptyset$. Recall that inside $X_D$ apart from the perturbed region, $\omega_X=dx\wedge dy+d(\frac{1}{2}x^2\beta_{m,n})$. We show in the following that outside of the perturbed region, the $x$-coordinate of the section contained in $X_D$ is locally constant, which obviously leads to a contradiction.

To see this fact, we write the one-form $\frac{1}{2}x^2\beta_{m,n}$ as
\begin{equation}
  \frac{1}{2}x^2\beta_{m,n}=fds+gdt
\end{equation}
where $(s,t)$ is the local coordinate for $B_0$. We next compute that the horizontal lifts $\partial_s^\#$, $\partial_t^\#$ of the two vector fields $\partial_s$, $\partial_t$ are:
\begin{equation}
  \partial_s^\# = \partial_s +\frac{\partial f}{\partial y}\partial_x-\frac{\partial f}{\partial x}\partial_y
\end{equation}
\begin{equation}
  \partial_t^\# = \partial_t +\frac{\partial g}{\partial y}\partial_x-\frac{\partial g}{\partial x}\partial_y.
\end{equation}
It follows that if $u$ is horizontal, then the part of $u(s,t)=(s,t, x(s,t),y(s,t))$ in $X_D$ outside of the perturbed region satisfies:
\begin{equation}
  \frac{\partial x}{\partial s}=\frac{\partial f}{\partial y},\ \frac{\partial y}{\partial s}=-\frac{\partial f}{\partial x}
\end{equation}
\begin{equation}
  \frac{\partial x}{\partial t} = \frac{\partial g}{\partial y},\ \frac{\partial y}{\partial t}=-\frac{\partial g}{\partial x}.
\end{equation}
Recall that by our assumption, away from the perturbed region inside $X_D$, we have $f_y=g_y=0$. It follows that $x$ is locally constant.

This concludes the proof of Lemma \ref{no crossing lemma}.
\end{proof}

A parallel result holds for $X^{m,n}$:
\begin{lemma}\label{coproduct no crossing lemma}
Assume $J$ is a fibration-compatible almost complex structure on $X^{m,n}$. If $C$ is a $J$-holomorphic section of $X^{m,n}$ such that all wrapping numbers are zero, then $C$ is either contained in $X_H$ or contained in $X_D$.
\end{lemma}
The proof of this result, however, is different from the one described above, so we present the details here:

\begin{proof}[Proof of Lemma \ref{coproduct no crossing lemma}]
As before, we only need to consider the case where $C$ is not horizontal. Suppose there is some $J$-holomorphic section $C$ that is neither contained in $X_H$ nor $X_D$, WLOG we assume that there is some $\epsilon_1\in(\delta, \epsilon)$ such that $C$ intersects both $F_{\pm\epsilon_1}$ transversely and $C\bigcap F_{\pm\epsilon_1}\ne\emptyset$.

Let $(p^{\pm},q_1^{\pm},q_2^{\pm})$ denote the homology classes of $C\bigcap F_{\pm\epsilon_1}$. Since $\eta(C)=0$, we again observe that $p^{\pm}=0$. Now the local energy inequality implies that:
\begin{equation}
m q_1^{+}+n q_2^{+}\ge 1,\ m q_1^{-}+n q_2^{-}\le -1.
\end{equation}
Let $d_\infty$(resp. $d_1$, $d_2$) $\in\{0,1\}$ denote the number of punctures of $C_{[-\epsilon_1,\epsilon_1]}$ that project to the positive puncture (resp. the two negative punctures) of $B_0$. We have:
\begin{equation}
d_1(0,1,0)+d_2(0,0,1)+(0,q_1^-,q_2^-)=d_{\infty}(0,1,1)+(0,q_1^+,q_2^+)
\end{equation}
and hence
\begin{equation}
2\le  m(q_1^+-q_1^-)+n(q_2^+-q_2^-) = m(d_1-d_\infty)+n(d_2-d_\infty).
\end{equation}

We conclude that $d_\infty=0$ and that at least one of $d_1$ and $d_2$ is $1$. There are two possibilities:
\begin{enumerate}
    \item $d_1=d_2=1$. If this is the case, then $C$ does not have other outputs. We conclude that for any small enough $\epsilon_2$, the section $C$ cannot intersect both $F_{1\pm\epsilon_2}$, otherwise the exact same argument would tell us that $C$ has at least another output over $x=1$. Choose $l\in \{1-\epsilon_2,1+\epsilon_2\}$ such that $C\bigcap F_l=\emptyset$. We now look at $C\bigcap F_{[\epsilon_1,l]}$. This part of $C$ can only have a positive puncture (or no punctures at all) with homology class $(-k,1,1)$ for some $k\in\{1,2,\cdots, m+n\}$, but the same homology class should match $(0,q_1^+,q_2^+)$, which means that there's no positive puncture. So we conclude that $C\bigcap F_{[\epsilon_1,l]}$ is a surface without puncture, whose boundary is $-C\bigcap F_{\epsilon_1}$, which implies that $q_1^+=q_2^+=0$, contradicting $m q_1^{+}+n q_2^{+}\ge 1$.

    \item We have either $d_1=1$ and $d_2=0$ or $d_1=0$ and $d_2=1$. WLOG let us assume the first case happens. There are two sub-cases.

    \textbf{Case 2.1} If there is some small $\epsilon_2\in(\delta,\epsilon)$ such that $C\bigcap F_{1-\epsilon_2}=\emptyset$ or $C\bigcap F_{1+\epsilon_2}=\emptyset$, then as before we fix $l\in\{1-\epsilon_2,1+\epsilon_2\}$ such that $C\bigcap F_l=\emptyset$, and look at $C\bigcap F_{[\epsilon_1,l]}$. This part of $C$ can have at most one positive puncture with homology class $(-k,1,1)$ where $k\in\{1,2,\cdots, m+n\}$ and at most one negative puncture with homology class $(-j,0,1)$ for some $j\in\{1,2,\cdots, n\}$. If $C\bigcap F_{[\epsilon_1,l]}$ has no punctures, then we argue as before to show that $q_1^+=q_2^+=0$, which leads to a contradiction. So $C\bigcap F_{[\epsilon_1,l]}$ has at least one puncture, but then again by homology considerations we conclude that $C\bigcap F_{[\epsilon_1,l]}$ has precisely two punctures, with homology classes $(-k,1,1)$ and $(-k,0,1)$ for some $k\in\{1,2,\cdots, n\}$. Now we have:
    \begin{equation}
    (-k,1,1)=(0,q_1^+,q_2^+)+(-k,0,1)
    \end{equation}
    which implies that $q_1^+=1$ and $q_2^+=0$. Now $d_1=1$ and $d_2=0$ tells us that $q_1^-=q_1^+-1=0$ and $q_2^-=q_2^+=0$, contradicting $m q_1^{-}+n q_2^{-}\le -1$.

    \textbf{Case 2.2} The other possibility is that we can find some $\epsilon_2\in(\delta,\epsilon)$ such that $C$ intersects both $F_{1\pm\epsilon_2}$ transversely. We use $(p^{1\pm\epsilon_2},q_1^{1\pm\epsilon_2},q_2^{1\pm\epsilon_2})$ to denote the homology classes of $C\bigcap F_{1\pm\epsilon_2}$. The condition $\eta(C)=0$ now translates to $p^{1\pm\epsilon_2}+m q_1^{1\pm\epsilon_2}+ n q_2^{1\pm\epsilon_2} = 0$, because the wrapping number is now the integral of $dy_R$, which equals the integral of $dy+\beta^{m,n}$. The local energy inequality tells us that:
    \begin{equation}
    p^{1-\epsilon_2}+(1-\epsilon_2)(m q_1^{1-\epsilon_2}+ n q_2^{1-\epsilon_2})>0\end{equation}\begin{equation} p^{1+\epsilon_2}+(1+\epsilon_2)(m q_1^{1+\epsilon_2}+ n q_2^{1+\epsilon_2})>0
    \end{equation}
    which simplifies to
    \begin{equation}
    m q_1^{1+\epsilon_2}+n q_2^{1+\epsilon_2}\ge 1,\ m q_1^{1-\epsilon_2}+n q_2^{1-\epsilon_2}\le -1.
    \end{equation}
    Let $d'_\infty$(resp. $d'_2$) $\in\{0,1\}$ denote the number of punctures of $C_{[1-\epsilon_2,1+\epsilon_2]}$ that project to the positive puncture (resp. the second negative punctures) of $B_0$. We have:
    \begin{equation}
    d'_2(-n,0,1)+(p^{1-\epsilon_2},q_1^{1-\epsilon_2},q_2^{1-\epsilon_2})=d'_{\infty}(-m-n,1,1)+(p^{1+\epsilon_2},q_1^{1+\epsilon_2},q_2^{1+\epsilon_2})
    \end{equation}
    which implies that
    \begin{equation}
    q_1^{1+\epsilon_2}=q_1^{1-\epsilon_2}-d_\infty',\ q_2^{1+\epsilon_2}=q_2^{1-\epsilon_2} +d'_2-d_\infty'.
    \end{equation}
    Again $2\le  m(q_1^{1+\epsilon_2}-q_1^{1-\epsilon_2})+n(q_2^{1+\epsilon_2}-q_2^{1-\epsilon_2})$ tells us that $d'_\infty=0$ and $d'_2=1$, and hence $p^{1-\epsilon_2}-n=p^{1+\epsilon_2}$.

    We now look at $C\bigcap F_{[-\epsilon_1,1+\epsilon_2]}$. This part of $C$ has two outputs with the homology classes $(0,1,0)$ and $(-n,0,1)$, and at most one puncture with homology class $(-k,1,1)$ for some $k\in\{1,2,\cdots, m+n-1\}$. We also have $\partial (C\bigcap F_{[-\epsilon_1,1+\epsilon_2]}) = C\bigcap F_{1+\epsilon_2} - C\bigcap F_{-\epsilon_1}$. We observe that $C\bigcap F_{[-\epsilon_1,1+\epsilon_2]}$ must contain a positive puncture, otherwise $p_1^{1+\epsilon_2}=-n$, and hence $p^{1-\epsilon_2}=0$, so $\eta=0$ implies that $m q_1^{1-\epsilon_2}+ n q_2^{1-\epsilon_2}=0$, contradicting the local energy inequality $m q_1^{1-\epsilon_2}+ n q_2^{1-\epsilon_2}\le-1$.
    Finally, we have:
    \begin{equation}
    (0,q_1^-,q_2^-)+(0,1,0)+(-n,0,1)=(p^{1+\epsilon_2},q_1^{1+\epsilon_2},q_2^{1+\epsilon_2})+(-k,1,1)
    \end{equation}
    which implies that $q_1^-=q_1^{1+\epsilon_2}$ and $q_2^-=q_2^{1+\epsilon_2}$, but then the local energy inequalities $m q_1^{-}+n q_2^{-}\le -1$ and $m q_1^{1+\epsilon_2}+n q_2^{1+\epsilon_2}\ge 1$ cannot both be true. This concludes the proof of Lemma \ref{coproduct no crossing lemma}.
\end{enumerate}
\end{proof}

\section{The ``no crossing'' results for general \texorpdfstring{$J$}{J}}\label{general no crossing}
Although the fibration compatible almost complex structures in Definition \ref{def:fibration_compatible_J} are convenient to work with, they are not suitable for defining the cobordism map. The reason is that for given fibration compatible $J$, not all $J$-holomorphic sections are cut out transversely, so there is not a well defined count for the cobordism map as defined in Section \ref{setup}. In this section, we use the SFT compactness theorem developed in \cite{bourgeois2003compactness} to show that we can always perturb the almost complex structure slightly to a tame almost complex structure - not necessarily fibration compatible, in such a way that the no crossing results Lemma \ref{no crossing lemma} and Lemma \ref{coproduct no crossing lemma} continue to hold.

Throughout this section, we let $X$ denote either the bundle $X_{m,n}$ or $X^{m,n}$. We fix a fibration compatible almost complex structure $J$ on $X$, and denote by $J_+^1$, $J_+^2$, $J_-$ its restrictions on the three cylindrical ends of $X$. In the case that $\partial \Sigma\ne \emptyset$, we choose coordinates $(x_i,y_i)$ near each boundary component of $\partial \Sigma$, such that any almost complex structure we choose, even if it is not fibration compatible elsewhere, is fibration compatible near the boundary and sends $\partial_{x_i}$ to $\partial_{y_i}$.

\begin{theorem}\label{classification for general J}
Let $\{J_k\}$ be a sequence of tame almost complex structures that $C^\infty$ converges to a fixed fibration-compatible almost complex structure $J$, and $\{C_k\}$ be a sequence of finite-energy $J_k$-holomorphic sections, which we view as maps $u_k: B_0\to X$, that are asymptotic to fixed Reeb orbits in $Y_{\phi^m}$, $Y_{\phi^n}$ and $Y_{\phi^{m+n}}$. If all wrapping numbers of $\{C_k\}$ vanish, then $C_k$ is contained in $X_H$ or $X_D$ for sufficiently large $k$.
\end{theorem}

The proof of Theorem \ref{classification for general J} relies largely on a careful analysis of $J$-holomorphic sections in $X$ and the symplectizations $Y_{\phi^m}$, $Y_{\phi^n}$ and $Y_{\phi^{m+n}}$, which we take up in the following subsections. To begin the proof of Theorem \ref{classification for general J}, let us make the following simple observation. We could slightly shrink the two open subsets $X_H$ and $X_D$ to $X_{H,\Tilde{\epsilon}}$ and $X_{D,\Tilde{\epsilon}}$, where $\Tilde{\epsilon}\in (\delta, \epsilon)$ and
\begin{equation}
    X_{D,\Tilde{\epsilon}}:=B_0\times (-\Tilde{\epsilon},1+\Tilde{\epsilon})_x\times S^1_y
\end{equation}
\begin{equation}
    X_{H,\Tilde{\epsilon}}:=B_0\times (\Sigma-(\Tilde{\epsilon},1-\Tilde{\epsilon})_x\times S^1_y)
\end{equation}
such that Lemma \ref{no crossing lemma} and Lemma \ref{coproduct no crossing lemma} still hold for the new cover $X=X_{D,\Tilde{\epsilon}}\cup X_{H,\Tilde{\epsilon}}$.

\subsection{ \texorpdfstring{$J$}{J}-holomorphic cylinders in symplectizations}
The next step is to analyze $J$-holomorphic cylinders in the symplectization $\R\times Y_{\phi^m}$. The analysis for the remaining cases of $Y_{\phi^n}$ and $Y_{\phi^{m+n}}$ are analogous. Similar to what we saw in Section \ref{setup}, there is a decomposition of $Y_{\phi^m}$:
\begin{equation}
    Y_{m,D,\Tilde{\epsilon}}:=S^1_t\times (-\Tilde{\epsilon},1+\Tilde{\epsilon})_x\times S^1_y
\end{equation}
\begin{equation}
   Y_{m,H,\Tilde{\epsilon}}:=S^1_t \times (\Sigma-(\Tilde{\epsilon},1-\Tilde{\epsilon})_x\times S^1_y).
\end{equation}
When $\Tilde{\epsilon} = \epsilon$, without causing confusions, we will abbreviate the two components by $Y_D$ and $Y_H$ respectively. The gluing map of $Y_{m,D,\Tilde{\epsilon}}$ and $Y_{m,H,\Tilde{\epsilon}}$ is defined similarly as in Section \ref{setup}. For $J$-holomorphic sections in $\R\times Y_{\phi^m}$, the wrapping numbers are defined similarly, see \cite{hutchings2005periodic} Definition 4.2.

As in Section \ref{section:no_crossing}, by slightly abusing the notations, let us denote by $F_x$ the three-dimensional manifold $\R\times S^1_t \times \{x\}\times S^1_y\subset Y_D$. Let $F_{(x_1,x_2)}$ denote the four-manifold $\R\times S^1_t \times (x_1,x_2)_x\times S^1_y\subset Y_D$.
We also identify the first homology class in $Y_D$ with a pair $(p,q)$. Fix a symplectization compatible (and hence by Definition \ref{def:fibration_compatible_J}, a fibration compatible) almost complex structure $J$ on $\R\times Y_{\phi^m}$. The local energy inequality for $J$-holomorphic sections in $\R\times Y_{\phi^m}$ is the following:
\begin{lemma}[\cite{hutchings2005periodic} Lemma 3.11] \label{local energy symp}
Let $C$ be a $J$-holomorphic section, which we write as a map $u: \R\times S^1\to \R\times Y_{\phi^m}$. Assume that $C$ intersects $F_x$ transversely for some $x$ in the unperturbed range. We have:
\begin{equation}
    p+mxq\ge 0.
\end{equation}
Furthermore, the equality holds if and only if $C\cap F_x=\emptyset$.
\end{lemma}
\begin{proof}
This is a straightforward generalization of Lemma \ref{local energy lemma}. For a different proof, see \cite{hutchings2005periodic}.
\end{proof}

The above inequality implies the following ``no crossing'' result for $J$-holomorphic cylinders in $ \R\times Y_{\phi^m}$:
\begin{lemma}\label{symp cylinders}
Let $J$ be a symplectization-compatible almost complex structure on $ \R\times Y_{\phi^m}$. If $C$ is a $J$-holomorphic section of the bundle $\R\times Y_{\phi^m}\to \R\times S^1_t$ with vanishing wrapping numbers, then:
\begin{enumerate}
\item For any $\Tilde{\epsilon}\in (\delta,\epsilon)$, the section $C$ is either contained in $\R\times Y_{m,D,\Tilde{\epsilon}}$ or $\R\times Y_{m,H,\Tilde{\epsilon}}$;
\item For such a section, if the positive end is one of the two orbits over $x=0$ (resp. $x=1$), then for any $\Tilde{\epsilon}\in (\delta,\epsilon)$, $C$ is contained in $F_{(-\Tilde{\epsilon},\Tilde{\epsilon})}$ (resp. $F_{(1-\Tilde{\epsilon},1+\Tilde{\epsilon})}$);
\item For such a section, if the negative end is one of the two orbits over $x=0$ (resp. $x=1$), then for any $\Tilde{\epsilon}\in (\delta,\epsilon)$,  the section $C$ is contained in $\R\times Y_{m,H,\Tilde{\epsilon}}$;
\item Finally, if such a section does not have any end over $x=0$ or $x=1$, then it is completely contained in $\R\times (Y_D-Y_H)$ or $\R\times (Y_H-Y_D)$.
\end{enumerate}
\end{lemma}

The proof is similar to that of Lemma \ref{no crossing lemma}, but it is worthwhile to write down the details.
\begin{proof}
For any $\epsilon_1\in (\delta,\epsilon)$, let us denote the homology classes of $C\cap F_{\pm \epsilon_1}$ by $(p^\pm, q^\pm)$ (the choice of $\epsilon_1$ does not matter here). Since all wrapping numbers of $C$ vanish, we conclude that $p^\pm = 0$. 

To prove the first bullet point, suppose $C$ is not contained in either region. WLOG we could assume there is some $\epsilon_1$ such that  $C\cap F_{\pm \epsilon_1}\ne\emptyset$. Now Lemma \ref{local energy symp} tells us that
\begin{equation}
    m\epsilon_1 q^+>0>m\epsilon_1 q^-.
\end{equation}
So $q^+\ge 1$ and $q^-\le -1$. Notice that for punctures of $C$ that are contained in $F_{[-\epsilon_1,\epsilon_1]}$, the homology class is $(0,1)$. Let us assume there are $d_\infty\in\{0,1\}$ (resp. $d_{-\infty}$) many of such positive (resp. negative) punctures, and we have:
\begin{equation}
    d_{-\infty}(0,1)+(0,q^-) = d_{\infty}(0,1)+(0,q^+).
\end{equation}
But this is not possible, because otherwise
\begin{equation}
    2\le q^+-q^-=d_{-\infty}-d_{\infty}\le 1.
\end{equation}

To prove the second bullet point, it suffices to show that for any $\epsilon_1\in (\delta,\epsilon)$, we have $C\cap F_{\pm \epsilon_1}=\emptyset$ and $C\cap F_{1\pm \epsilon_1}=\emptyset$. WLOG suppose $C\cap F_{ \epsilon_1}\ne \emptyset$ or $C\cap F_{-\epsilon_1}\ne \emptyset$. By the same argument as in the previous paragraph, we have $q^+-q^-\ge 1$, but now we have $d_\infty = 1$, so
\begin{equation}
    1\le q^+-q^- = d_{-\infty}-d_{\infty} \le 0,
\end{equation}
a contradiction.

To prove the third bullet point, simply notice that otherwise such a section is completely contained in $\R\times Y_D$ by the first bullet point. Now observe that Reeb orbits in $Y_D$ that are over different values of $x$ have different homology classes, it follows that both ends of $C$ are over $x=0$ or $x=1$. Now the second bullet point shows that such a section is contained in $\R\times Y_{m,H,\Tilde{\epsilon}}$ as well.

Finally, to prove the last bullet point, observe that (in the same notation as before) $d_{\pm\infty}=0$ forces that $q^\pm=0$, hence $C\cap F_{\pm \epsilon_1}=\emptyset$. Similarly $C\cap F_{1\pm \epsilon_1}=\emptyset$ for any $\epsilon_1$.
\end{proof}

\subsection{More about \texorpdfstring{$J$}{J}-holomorphic sections in the twist region}
To prove Theorem \ref{classification for general J}, the final ingredient we need is a more detailed understanding of $J$-holomorphic sections that are contained in the twist region $X_D$. Let us recall that, for $J$-holomorphic sections of $X_{m,n}$ that are contained in the twist region $X_D=B_0\times (-\epsilon, 1+\epsilon)_x\times S^1_y$, the Reeb orbits can occur over:
\begin{enumerate}
    \item $x=\frac{i}{m}$ ($i\in\{0,1,\cdots, m\}$) for the first positive end;
    \item $x=\frac{j}{n}$ ($j\in\{0,1,\cdots, n\}$) for the second positive end;
    \item $x=\frac{k}{m+n}$ ($k\in\{0,1,\cdots, m+n\}$) for the negative end.
\end{enumerate}
The next lemma tells us that for pseudo-holomorphic sections of $X_{m,n}$ that are contained in the twist region, the three ends must in fact lie over the same $x$-coordinate.

\begin{lemma}\label{more in twist}
Let $J$ be a fibration-compatible almost complex structure on $X_{m,n}$, and $C$ be a $J$-holomorphic section that is completely contained in $X_D=B_0\times (-\epsilon, 1+\epsilon)_x\times S^1_y$. Then the $x$ coordinate of the three cylindrical ends of $C$ asymptote to the same value. Furthermore, $C$ itself is completely contained in the $\delta$-neighborhood of the slice $F_x$ (the subscript $x$ denotes the $x$ value to which the ends of $C$ asymptote.)
\end{lemma}

\begin{proof}
The Reeb vector field near the first positive end is $\partial_t-mx\partial_y$, so the homology class of any Reeb orbit over $x=\frac{i}{m}$ is $[S^1_{t_1}]-i[S^1_y]\in H_1(X_D)$. Similarly, for Reeb orbits over the second positive end with the $x$-coordinate $\frac{j}{n}$, the homology class is $[S^1_{t_2}]-j[S^1_y]\in H_1(X_D)$; the homology class for Reeb orbits over the negative end with the $x$-coordinate $\frac{k}{m+n}$ is $[S^1_{t_1}]+[S^1_{t_2}]-k[S^1_y]\in H_1(X_D)$.

It follows that for a $J$-holomorphic section that is completely contained in $X_D$, we have $k=i+j$ for homological reasons. Now if the three ends don't share the same $x$-coordinates, WLOG we can assume that $\frac{i}{m}<\frac{i+j}{m+n}<\frac{j}{n}$. Pick some $x_0\in(\frac{i+j}{m+n},\frac{j}{n})$ such that $C$ intersects the slice $F_{x_0}$ transversely, then the homology class of $C\bigcap F_{x_0}$ in $H_1(X_D)$ is $(-j, 0, 1)$. Using Lemma \ref{local energy lemma}, we have
\begin{equation}
-j+x_0(m\cdot 0+n\cdot 1)\ge 0
\end{equation}
which implies that $x_0\ge \frac{j}{n}$, a contradiction.

Now suppose $C$ is not contained in the $\delta$-neighborhood of the slice $F_x$, we can choose some $\Tilde{\epsilon}$ slightly bigger than $\delta$ such that $C$ intersects $F_{x\pm \Tilde{\epsilon}}$ transversely and the intersect is nonempty. But notice that  $C\cap F_{x\pm \Tilde{\epsilon}}$ are both null-homologous, so this is a violation of Lemma \ref{local energy lemma}.
\end{proof}

\begin{remark}
The second part of the above Lemma also holds for $J$-holomorphic sections of $X^{m,n}$ that are completely contained in the twist region. Namely, if all three ends of such a section share the same $x$-coordinate, then the entire section is contained in the $\delta$-neighborhood of the slice $F_x$.
\end{remark}

\subsection{Proof of ``No-crossing'' for general $J$ (Theorem \ref{classification for general J})}
Now we are ready to prove the main result of this section.

\begin{proof}[Proof of Theorem \ref{classification for general J}]
Fix some $\Tilde{\epsilon}\in(\delta, \epsilon)$. Suppose that the statement of Theorem \ref{classification for general J} fails, by the SFT compactness theorem, we can find a subsequence of $\{C_k\}$, still denoted by $\{C_k\}$, such that :
\begin{enumerate}
    \item For every $k$, $C_k$ is not contained in $X_H$ or $X_D$, and
    \item $\{C_k\}$ converges to a $J$-holomorphic building $\mathcal{B}$.
\end{enumerate}
Let us first observe that by our assumptions, $\pi_2(X_{m,n})$, $\pi_2(X^{m,n})$, $\pi_2(Y_{\phi^m})$ are all trivial, so bubbling off of $J$- holomorphic spheres can not occur in any level of $\mathcal{B}$. Notice also that it is not possible for any component of any level of $\mathcal{B}$ to have only positive or only negative punctures, simply by homological considerations. The above two observations imply that 
\begin{enumerate}
\item $\mathcal{B}$ has no nodes;
\item The main level of $\mathcal{B}$ is a $J$-holomorphic section of $X$, and 
\item Every other level of $\mathcal{B}$ is a (resp. pair of) holomorphic cylinder in the symplectization of $Y_{\phi^{m+n}}$ (resp. $Y_{\phi^{m}}\coprod Y_{\phi^{n}}$).
\end{enumerate}
We note that all levels of $\mathcal{B}$ must have vanishing wrapping numbers. The reason is that the wrapping number is homological, so the sum of the wrapping number from all different level is equal to zero. By Remark \ref{remark_nonnegative_wrapping_number}, all wrapping numbers are non-negative, so they have to vanish in each level as well.

If the main level of $\mathcal{B}$ is contained in $X_{H,\Tilde{\epsilon}}$, we can use Lemma \ref{symp cylinders} and induction to show that all other levels of $\mathcal{B}$ are contained in $\R\times Y_{m+n,H,\Tilde{\epsilon}}$ (or $\R\times (Y_{m,H,\Tilde{\epsilon}}\coprod Y_{n,H,\Tilde{\epsilon}})$ respectively). For example, the first level above the main level consists of one $J_+$-holomorphic cylinder in $\R\times Y_{\phi^{m+n}}$ (if $X=X^{m,n}$) or a pair of $J_+$-holomorphic cylinders in $\R\times (Y_{\phi^m}\coprod Y_{\phi^n})$ (if $X=X_{m,n}$). In either case, those $J_+$-holomorphic cylinders have negative ends which are either over $x=0,1$ inside the twist regions, or outside the twist region. Now the third and forth bullets points of Lemma \ref{symp cylinders} tell us that these cylinders are entirely contained in $\R\times Y_{m+n,H,\Tilde{\epsilon}}$ or $\R\times (Y_{m,H,\Tilde{\epsilon}}\coprod Y_{n,H,\Tilde{\epsilon}})$. We conclude, using induction, that all levels above the main level are contained in the same region. Now let us consider the first level under the main level. For any such $J_-$-holomorphic cylinder, if the positive end is over $x=0,1$, then by the second bullet point of Lemma \ref{symp cylinders}, they are completely contained in $F_{(-\Tilde{\epsilon},\Tilde{\epsilon})}$; if the positive end is contained outside of the twist regions, then the third and forth bullet points of Lemma \ref{symp cylinders} imply that the cylinders are contained in $\R\times Y_{m+n,H,\Tilde{\epsilon}}$ or $\R\times (Y_{m,H,\Tilde{\epsilon}}\coprod Y_{n,H,\Tilde{\epsilon}})$. Again, we can repeat the above analysis to find that all levels under the main level are contained in $\R\times Y_{m+n,H,\Tilde{\epsilon}}$ or $\R\times (Y_{m,H,\Tilde{\epsilon}}\coprod Y_{n,H,\Tilde{\epsilon}})$. In summary, the entire building is contained in the (slightly shrunk) non-twist region, which implies that for sufficiently large $k$, the section $C_k$ is contained in $X_H$ as well, a contradiction.

If the main level of $\mathcal{B}$ is contained in $X_{D,\Tilde{\epsilon}}-X_{H,\Tilde{\epsilon}}$, then again we can use the fourth bullet point of Lemma \ref{symp cylinders} and induction to deduce that all other levels of $\mathcal{B}$ are contained in $\R\times Y_{m+n,D,\Tilde{\epsilon}}$ or $\R\times (Y_{m,D,\Tilde{\epsilon}}\coprod Y_{n,D,\Tilde{\epsilon}})$. It follows that for sufficiently large $k$, the section $C_k$ is completely contained in $X_D$, a contradiction.

\end{proof}

\section{The product}\label{the product}
In this section, we use the no-crossing results to calculate the pair-of-pants product defined in Section \ref{setup}: \begin{equation}\label{product hom}
    HF_*(\phi^m)\otimes HF_*(\phi^n)\longrightarrow HF_*(\phi^{m+n})
\end{equation}
where $\phi$ is the (Hamiltonian perturbed) positive Dehn twist along a homologically nontrivial simple closed curve $\gamma\subset \Sigma$, with the extra conditions stated in Theorem \ref{product for multiple twists}.

As reviewed in Section \ref{setup}, we fix the cobordism $X=X_{m,n}$ and a generic Hamiltonian perturbation. We always assume that the almost complex structure $J$ is $C^\infty$ close to a fibration-compatible one, as in Section \ref{section:no_crossing}. Furthermore, we require that near each boundary component of $\partial\Sigma$ with local coordinates $(x_i,y_i)$, the almost complex structure $J$ is fibration-compatible, and is induced from the almost complex structure on $\Ver$ that sends $\partial_{x_i}$ to $\partial_{y_i}$. Lemma \ref{mp} tells us that $J$-holomorphic sections cannot approach $\partial\Sigma$ by the maximum principle.

\subsection{Several remarks on monotonicity}
To define fixed point Floer homology without using the Novikov rings, we need a monotonicity condition. In what follows, we will use a slightly stronger version of ``weak monotonicity'' introduced in \cite{cotton2009symplectic}.

\begin{definition}\label{monotone def}
(\cite{cotton2009symplectic} Condition 2.5) Let $\psi$ be a symplectomorphism of $(\Sigma, \omega_0)$. Let $\omega_\psi$ denote the 2-form on the mapping torus $Y_\psi$ induced by $\omega_0$, and $\Ver$ the vertical distribution of $Y_\psi\to S^1$. We say $\psi$ is weakly monotone if $[\omega_\psi]$ vanishes on the kernel of
\begin{equation}
    c_1(\Ver): H_2(Y_\psi)\longrightarrow\R.
\end{equation}
\end{definition}

We have the following:
\begin{lemma}\label{symp monotone}
For the positive Dehn twist $\phi: (\Sigma,\omega_0)\to (\Sigma,\omega_0)$, the map $\phi^m$ is weakly monotone for any positive integer $m$.
\end{lemma}

\begin{proof}
The proof is almost verbatim to that of \cite{hutchings2005periodic} Lemma 5.1, the only difference is that in our setting $\langle [\Sigma],c_1(\Ver)\rangle = 2-2g(\Sigma)$ if $\partial\Sigma=\emptyset$.
\end{proof}

The above lemma tells us that the count in (\ref{count}) is finite, so the fixed point Floer homology $HF_*(\phi^m)$ is well-defined without using the Novikov rings. Similarly, we need a weak monotonicity condition for the count (\ref{cobordismcount}) to be finite. 

\begin{definition}\label{def: bundle monotone}
Let $\pi: (E,\omega)\to B$ be a symplectic fiber bundle, and $\Ver:=\Ker (d\pi)$ be the vertical distribution. We say the symplectic bundle is weakly monotone if $[\omega]$ vanishes on the kernel of 
\begin{equation}
    c_1(\Ver): H_2(E)\longrightarrow \R.
\end{equation}
\end{definition}

Similarly, we have the following lemma, which tells us that the count (\ref{cobordismcount}) is finite\footnote{The reason is that the weak monotonicity condition ensures that pseudo-holomorphic sections with the same asymptotes and Fredholm index have the same vertical energy $\int_C \omega_X$, and hence we have an a priori bound on the total energy on such sections.}, and hence the product and coproduct structures induced by $X_{m,n}$ and $X^{m,n}$ are well-defined without use of Novikov rings. 

\begin{lemma}\label{cobordism monotone}
If $\partial\Sigma\ne\emptyset$ or $\Sigma$ is closed with genus at least 2, then both $X_{m,n}$ and $X^{m,n}$ are weakly monotone.
\end{lemma}

\begin{proof}
Take a closed surface $C\subset X$ such that $[C]$ lies in the kernel of $c_1(\Ver)$. Using the same notation as in Lemma \ref{local energy lemma}, let $(p,q_1,q_2)$ denote the homology class of $[C\cap F_0]=[C\cap F_1]\in H_1(X_D)$ (isotope $C$ slightly to make the two intersections transverse).  

Let us start with the situation where $\gamma$ is non-separating and $\Sigma$ is closed. It is not difficult to see that:
\begin{equation}\label{eta-chern relation}
    \langle [C], c_1(\Ver) \rangle= (2-2g(\Sigma))\eta(C)
\end{equation}
where $\eta$ is the wrapping number. We conclude that the wrapping number of $C$ is zero. As explained in the proof of Lemma \ref{no crossing lemma}, we have
\begin{equation}
    \eta(C)=p=p+mq_1+nq_2=0.
\end{equation}

Now if $\gamma$ is separating and $\Sigma$ is closed, let $\Sigma_1$, $\Sigma_2$ denote the two components of $\Sigma-[0,1]_x\times S^1_y$. By our assumption: $g(\Sigma_1),g(\Sigma_2)\ge 1$. Similar to the above, we have:
\begin{equation}
    \langle [C], c_1(\Ver) \rangle= (1-2g(\Sigma_1))\eta_1(C)+(1-2g_2(\Sigma_2))\eta_2(C).
\end{equation}
So $[C]\in \Ker(c_1(\Ver))$ implies that both wrapping numbers of $C$ vanish, hence we have $p=p+mq_1+nq_2=0$ again. 

It is not difficult to calculate, using the explicit expression of $\omega_{X}$, that
$\int_C \omega_X$
is a linear combination of $p$ and $mq_1+nq_2$. So $[\omega_X]$ indeed vanishes  on $[C]$.

Finally, if $\partial\Sigma\ne \emptyset$, then the wrapping number of $C$ is automatically zero if $\gamma$ is non-separating. If $\gamma$ is separating, then:
\begin{enumerate}
    \item If both components of $\Sigma-[0,1]_x\times S^1_y$ contains at least one component of $\partial C$, then both wrapping numbers of $C$ automatically vanishes;
    \item If only one of the components of $\Sigma-[0,1]_x\times S^1_y$, say $\Sigma_2$, contains components of $\partial \Sigma$ (so $\eta_2$ vanishes), then we have
    \begin{equation}
        \langle [C], c_1(\Ver)\rangle =(1-2g(\Sigma_1))\eta_1(C).
    \end{equation}
    So $[C]\in ker(c_1(\Ver)$ implies that $\eta_1$ vanishes as well.
\end{enumerate}
The conclusion is that $p=p+mq_1+nq_2=0$ regardless. Using the exact same argument as above, we conclude that $[\omega_X]$ vanishes on $[C]$ as well.
\end{proof}

Lemma \ref{symp monotone} tells us that for any positive integer $m$, $HF_*(\phi^m)$ is well defined without use of Novikov coefficients. In fact, it is well-known (see for example \cite{seidel1996symplectic,hutchings2005periodic}) that:
\begin{equation}\label{HF isomorphism}
HF_*(\phi^m)\cong H_*(\Sigma_0;\Z_2)\oplus (\oplus_{i=1}^{m-1} H_*(S^1))
\end{equation}
where the $i$-th component of $\oplus_{i=1}^{m-1} H_*(S^1)$ comes from the Reeb orbits inside the Dehn twist region over $x=\frac{i}{m}$.

\subsection{All sections have vanishing wrapping numbers}
In this section we explain that in our setting, all $J$-holomorphic sections of $X_{m,n}\to B_0$ with Fredholm index $0$ have vanishing wrapping numbers. This observation will allow us to use the no crossing results from Section \ref{basic case} and \ref{general no crossing}.

\begin{theorem}\label{eta=0 lemma}
Let $J$ be an almost complex structure on $X_{m,n}$ that is close to a fibration-compatible one. Suppose
\begin{itemize}
    \item If $\gamma$ is non-separating, then $\partial \Sigma \ne \emptyset$ or $\Sigma$ is closed with genus at least $2$;
    \item If $\gamma$ is separating, then each component of $\Sigma-\gamma$ either contains a component of $\partial\Sigma$ or has genus at least 2.
\end{itemize} 
Then for any Fredholm index zero $J$-holomorphic section $C$ with cylindrical ends, $C$ has vanishing wrapping numbers. 
\end{theorem}

To prove Theorem \ref{eta=0 lemma}, let us recall the Fredholm index formula. Let $C$ be a $J$-holomorphic section in $X_{m,n}$ with positive asymptotes $\alpha_i$ and negative asymptote $\beta$, represented by a map $u: B_0\to X_{m,n}$. Fix a trivialization $\tau$ of the vertical distribution along each Reeb orbit, and denote by $\langle c_1^\tau(TX_{m,n}),[C]\rangle$ the first Chern number of the complex vector bundle $u^*TX_{m,n}$ over $B_0$ with respect to the trivialization $\tau$ and the natural splitting $TX_{m,n}|_{\gamma} \cong \Ver\bigoplus \mathbb{R}\langle R, \partial_s \rangle$ over the ends. Here $\R \langle R, \partial_s\rangle$ denotes the distribution spanned by the Reeb vector field and the symplectization direction. For each asymptotic orbit, let $\CZ_\tau$ be the Conley-Zehnder index with respect to $\tau$. We have the Fredholm index formula:
\begin{equation}\label{Fredholm product}
\ind(C) = 1 + 2 \langle c_1^\tau(TX_{m,n}),[C]\rangle+ \sum \CZ_\tau(\alpha_i) - \CZ_\tau(\beta).
\end{equation}
Notice that in our setting, the map $u$ is a section of the fibration $X_{m,n}\to B_0$, so $u^*TX_{m,n}$ naturally splits as $u^*TX_{m,n}\cong TB_0\bigoplus u^*\Ver$. In light of this splitting, we have:
\begin{equation}
\langle c_1^\tau(TX_{m,n}),[C]\rangle = -1+\langle c_1^\tau(\Ver),[C]\rangle.
\end{equation}
So the index formula can be rewritten as:
\begin{equation}
\ind(C) = -1 + 2 \langle c_1^\tau(\Ver),[C]\rangle+ \sum \CZ_\tau(\alpha_i) - \CZ_\tau(\beta).
\end{equation}
Now recall that the Reeb orbits can be divided into two types: those coming from critical points of $mH_0$, $nH_0$ or $(m+n)H_0$ outside of $N$ and those lying inside the twist region. There is a natural choice of the trivialization $\tau$ of the distribution $\Ver$ over these Reeb orbits: for the critical points of $H$, the distribution $\Ver$ can be identified with the tangent space $T\Sigma$ at the point; and over the Dehn twist region $N$; we can identify $\Ver$ with $TN$. We will always choose $\tau$ as above, and the Conley-Zehnder index $\CZ_\tau$ with respect to such a trivialization is:

\begin{enumerate}
    \item -1, if the orbit comes from a local minimum of $H$, or is an elliptic orbit inside the Dehn twist region;
    \item 0, if the orbit comes from a saddle point of $H$, or is a hyperbolic orbit inside the Dehn twist region;
    \item 1, if the orbit comes from a local maximum of $H$.
\end{enumerate}

Following \cite{hutchings2005periodic}, we now demonstrate a lemma relating the relative first Chern number $\langle c_1^\tau(\Ver),[C]\rangle$ to the wrapping number $\eta(C)$ (similar ideas were applied in the proof of Lemma \ref{cobordism monotone}; we present a proof of the generalization of equation (\ref{eta-chern relation}) here):

\begin{lemma}\label{relative eta-chern}
If $\Sigma$ is a closed surface with genus $g$, the loop $\gamma$ is non-separating, and $C$ is a $J$-holomorphic section, then \begin{equation}
\langle c_1^\tau(\Ver),[C]\rangle = (2-2g)\eta(C).
\end{equation}
\end{lemma}

\begin{proof}
Choose a generic point, which we denote by $pt\in \Sigma'$, such that $C$ intersects $B_0\times \{pt\}$ transversely. Recall that $\eta(C)$ is by definition the algebraic intersection number $\#C\bigcap(B_0\times \{pt\})$.

Choose a section $\psi$ of $\Ver$ over $X_{m,n}$, with the following property: 
\begin{enumerate}
    \item When restricted to the Reeb orbits, $\psi$ is constant with respect to $\tau$;
    \item There are $l$ points $p_1, p_2, \cdots, p_l$ concentrated in an arbitrarily small neighborhood of $pt$, such that on each fiber of $X_{m,n}\to B_0$, the section $\psi$ has transverse zeroes at precisely $p_1, p_2, \cdots, p_l$, with total degree $2-2g$.
\end{enumerate}

We can also arrange $\psi$ so that $C$ intersects each $B_0\times \{p_i\}$ transversely. Now by definition, $\langle c_1^\tau(\Ver),[C]\rangle$ is the algebraic count of zeroes of $u^*\psi$. Observe that the zeroes of $\psi|_C$ occurs at precisely $C\bigcap (B_0\times\{p_1, p_2, \cdots, p_l\})$, and the algebraic count of these zeroes is $(2-2g)\eta(C)$. 
\end{proof}

Now we are ready to prove Theorem \ref{eta=0 lemma}.

\begin{proof}[Proof of Theorem \ref{eta=0 lemma}]
Let us start with the case where $\gamma$ is non-separating. As remarked before, if $\partial\Sigma\ne\emptyset$, then $\eta(C)$ is automatically zero. If $\Sigma$ is closed, by equation (\ref{eta-chern relation}), we have the following:
\begin{equation}
    0=\ind(C)=-1 + 2(2-2g)\eta(C)+ \sum \CZ_\tau(\alpha_i) - \CZ_\tau(\beta).
\end{equation}
But since $\CZ_\tau\in\{-1,0,1\}$, we have
\begin{equation}
\ind (C)\le 2+ (4-4g)\eta(C).
\end{equation}
This, together with the fact that $\eta(C)\ge 0$ and the assumption $g\ge 2$, forces that $\eta(C)=0$.

Now let us deal with the case where $\gamma$ is separating. As before, let us denote by $\Sigma_1$ and $\Sigma_2$ the two components of $\Sigma-N$. If $\Sigma$ is closed, then similar to Lemma \ref{relative eta-chern}, we have:
\begin{equation}
    \langle c_1^\tau(\Ver),[C]\rangle = (1-2g(\Sigma_1))\eta_1(C)+(1-2g(\Sigma_2))\eta_2(C).
\end{equation}
So we have:
\begin{align}
\begin{split}
    0&=\ind(C)\\
    &=-1 + 2(1-2g(\Sigma_1))\eta_1(C)+2(1-2g(\Sigma_2))\eta_2(C)+ \sum \CZ_\tau(\alpha_i) - \CZ_\tau(\beta)\\
    &\le 2 + 2(1-2g(\Sigma_1))\eta_1(C)+2(1-2g(\Sigma_2))\eta_2(C).
\end{split}
\end{align}
By our assumption, $g(\Sigma_1),g(\Sigma_2)\ge2$. Combined with the fact that $\eta_1,\eta_2\ge 0$, we have
\begin{equation}
    \eta_1(C)=\eta_2(C)=0,
\end{equation}
as desired.

The case where both $\Sigma_1$ and $\Sigma_2$ contain a component of $\partial\Sigma$ is easy: we only need to observe as before that if $\Sigma_i$ contains a component of $\partial\Sigma$, then $\eta_i$ is automatically zero. The only remaining situation is the following: only one of the two components of $\Sigma-N$ contains a component of $\partial\Sigma$, and the other one has genus at least 2. WLOG let us assume $\Sigma_2$ is the one containing $\partial\Sigma$ (so $\eta_2$ vanishes automatically). Similar to what we saw in Lemma \ref{relative eta-chern}, we have the following:
\begin{equation}
    \langle c_1^\tau(\Ver),[C]\rangle = (1-2g(\Sigma_1))\eta_1(C)
\end{equation}
which implies that
\begin{align}
\begin{split}
    0&=\ind(C)\\
    &=-1 + 2(1-2g(\Sigma_1))\eta_1(C)+ \sum \CZ_\tau(\alpha_i) - \CZ_\tau(\beta)\\
    &\le 2 + 2(1-2g(\Sigma_1))\eta_1(C).
\end{split}
\end{align}
Again, since $g(\Sigma_1)\ge2$ and $\eta_1\ge 0$, we conclude that $\eta_1(C)$ has to vanish as well.

\end{proof}

\subsection{Computation of the product (proof of Theorem \ref{product for multiple twists})}
We are now ready to prove Theorem \ref{product for multiple twists}.
\begin{proof}[Proof of Theorem \ref{product for multiple twists}]
We choose a generic tame almost complex structure $J$ on $X_{m,n}$ such that all moduli spaces of Fredholm index zero sections are cut out transversely. We further assume that $J$ is $C^\infty$ close to a fibration-compatible almost complex structure so that Theorem \ref{classification for general J} applies. Theorem \ref{classification for general J} and Theorem \ref{eta=0 lemma} tell us that all the $J$-holomorphic sections are either contained in $X_H$ or $X_D$.

Now the count of $J$-holomorphic sections contained in $X_H$ precisely corresponds to the intersection product of $H_*(\Sigma_0;\Z_2)\subset HF_*(\phi^m)$ and $H_*(\Sigma_0;\Z_2)\subset HF_*(\phi^n)$ in the sense of the decomposition (\ref{HF isomorphism}). By results \cite{PSS,frauenfelder2007hamiltonian,lanzat2016hamiltonian}, the cobordism map of the pair-of-pants product in this case can be identified with the intersection pairing (notice that in our case $\pi_2(\Sigma_0)=0$, so no Novikov rings are needed here):
\begin{equation}
    H_*(\Sigma_0;\Z_2)\otimes H_*(\Sigma_0;\Z_2)\xrightarrow{\quad\cap\quad}H_*(\Sigma_0;\Z_2).
\end{equation}

To finish the proof, we only need to show that the count of sections contained in the twist region contributes to zero in the cobordism map. By Lemma \ref{more in twist}, any such $J$-holomorphic section must be contained in the $\delta$-neighborhood of some slice $F_x$, where $x=\frac{i}{d}$ for some $i\in \{1, 2\cdots, d-1\}$. Here $d:=\gcd(m,n)$. Recall that, near $x=\frac{i}{d}$, the symplectic fiber bundle is given by the trivial product $X_i:=B_0\times (\frac{i}{d}-\delta, \frac{i}{d}+\delta)_x\times S^1_y$ with the fiberwise symplectic closed 2-form $\omega_X$, which is a small Hamiltonian perturbation of $\omega_{X,0}=dx\wedge dy+d(\frac{1}{2}x^2\beta_{m,n})$. Now we use a change-of-coordinate trick to show that the above symplectic fiber bundle is equivalent to another one, which calculates the pair-of-pants product of the Hamiltonian Floer homology of \emph{small} Hamiltonians on $(\frac{i}{d}-\delta, \frac{i}{d}+\delta)_x\times S^1_y$. To do this, let $X_0$ be the trivial bundle $B_0\times (-\delta, \delta)_{x'}\times S^1_{y'}$ together with the fiberwise symplectic form $dx'\wedge dy'+d(\frac{1}{2}x'^2\beta_{m,n})$. Define a diffeomorphism $\mu: X_i\to X_0$ by:
\begin{equation}\label{eq: translation trick}
\begin{cases}
      x'=x-\frac{i}{d} \\
      y'=y+ i\cdot g'_{m,n}(p) 
    \end{cases}
\end{equation}
where $p$ denotes the coordinate on $B_0$. It's easy to see that $\mu$ preserves the fibers, and that $\mu$ pulls $dx'\wedge dy'+d(\frac{1}{2}x'^2\beta_{m,n})$ back to $\omega_{X,0}$, because (recall that $dg'_{m,n}=\frac{\beta_{m,n}}{d}$):
\begin{align}
\begin{split}
\mu^*(dx'\wedge dy'+d(\frac{1}{2}x'^2\beta_{m,n})) &= dx\wedge(dy+i\cdot g'_{m,n})+d(\frac{1}{2}(x-\frac{i}{d})^2\beta_{m,n})\\
 &= dx\wedge dy+\frac{i}{d} dx\wedge\beta_{m,n}+(x-\frac{i}{d})dx\wedge\beta_{m,n}\\&= dx\wedge dy+d(\frac{1}{2}x^2\beta_{m,n})\\&= \omega_{X,0}.
\end{split}
\end{align}

We can also push forward the chosen Hamiltonian perturbation to perturb $dx'\wedge dy'+d(\frac{1}{2}x'^2\beta_{m,n})$ on $X_0$ (we denote the perturbed fiberwise symplectic 2-form by $\omega'$). For any almost complex structure $J$ on $X_i$, we have the following one-to-one correspondence:
\begin{center}
    \{$J$-holomorphic sections of $X_i$\}$\xleftrightarrow{1:1}$\{$\mu_*(J)$-holomorphic sections of $X_0$\}.
\end{center}

Now, similar to the $J$-holomorphic sections that are contained in $X_H$, it is clear that sections contained in $X_0$ computes the pair-of-pants product of (a small perturbation of) the fixed points of time-1 maps of $m\cdot \frac{1}{2}x'^2$ and $n\cdot \frac{1}{2}x'^2$. This corresponds to the intersection product of $H_*((-\delta, \delta)_{x'}\times S^1_{y'})$, which is identically zero. This concludes the proof of Theorem \ref{product for multiple twists}.
\end{proof}

\section{The coproduct}\label{the coproduct}
In this section, we generalize the methods used in Section \ref{the product} further to compute the pair-of-pants coproduct of fixed point Floer homology of Dehn twists: $HF_*(\phi^{m+n})\to HF_*(\phi^m)\otimes HF_*(\phi^n)$. Note that Lemma \ref{cobordism monotone} tells us that the cobordism map is well-defined even without the use of Novikov rings. The goal of this section is to prove Theorem \ref{coproduct for multiple twists}. To begin with, similar to Theorem \ref{eta=0 lemma}, we have the following:

\begin{theorem} \label{cobordism eta=0}
Let $J$ be an almost complex structure on $X^{m,n}$ that is close to a fibration-compatible one. Suppose
\begin{itemize}
    \item If $\gamma$ is non-separating, then $\partial \Sigma \ne \emptyset$ or $\Sigma$ is closed with genus at least $2$;
    \item If $\gamma$ is separating, then each component of $\Sigma-\gamma$ either contains a component of $\partial\Sigma$ or has genus at least 2.
\end{itemize}
Then for any index zero $J$-holomorphic section $C$ with cylindrical ends, $C$ has vanishing wrapping numbers.
\end{theorem}
And the proof is almost the same as that of Theorem \ref{eta=0 lemma} (we only need to slightly modify the Conley-Zehnder index term). Combined with Theorem \ref{classification for general J}, it tells us the following:

\begin{cor}\label{coproduct sections classification}
Suppose the almost complex structure $J$, and the loop $\gamma$ satisfy the same condition as in Theorem \ref{cobordism eta=0}, then all $J$-holomorphic sections of $X^{m,n}\to B_0$ that have Fredholm index zero must be contained in $X_H$ or $X_D$. 
\end{cor}

What is different from Section \ref{the product} is that the count of sections contained in the twist region does not contribute to zero in the cobordism map. In the following two subsections, we first give a detailed understanding of the moduli space of all $J$-holomorphic sections in the Morse-Bott (unperturbed) setting, then explain what would happen if we perturb the form $\omega_{X,0}$ to break the Morse-Bott degeneracy.

\subsection{\texorpdfstring{$J$}{J}-holomorphic sections inside the twist region}\label{coproduct curves}
In this subsection, we analyze possible $J$-holomorphic sections inside $X_D$. Unless otherwise specified, the almost complex structure $J$ is assumed to be fibration-compatible, and when restricted to $\Ver$ sends $\partial_x$ to $\partial_y$ (so $J$ is completely determined by the fiberwise symplectic 2-form $\omega$). We start with the fiberwise symplectic 2-form $\omega_{X,0}=dx\wedge dy+d(\frac{1}{2}x^2\beta_{m,n})$ and view $(X_D,\omega_{X,0})$ as part of $\bar{X}_D:=B_0\times \R_x\times S^1_y$ with the same fiberwise symplectic 2-form. Note that by extending the cobordism we are not introducing new curves: Lemma \ref{local energy lemma} ensures that if the $x$-coordinates of the asymptotic Reeb vector fields of a given $J$-holomorphic section are contained in $(-\epsilon, 1+\epsilon)$, then the entire $J$-holomorphic section of $\bar{X}_D$ is in fact entirely contained in $X_D$.

Notice that with this Morse-Bott setting, the Reeb orbits at the ends come in $S^1$-families. The possible $x$-coordinates of such families are:
\begin{itemize}
    \item $x=\frac{k_\infty}{m+n}$ at the positive end,
    \item $x=\frac{k_1}{m}$ at the first negative end, and
    \item $x=\frac{k_2}{n}$ at the second negative end.
\end{itemize}
Observe that if a $J$-holomorphic section has cylindrical ends at $x=\frac{k_\infty}{m+n}, \frac{k_1}{m}$, and $\frac{k_2}{n}$, then $k_\infty=k_1+k_2$.

We make the following two basic observations about $J$-holomorphic sections of $\bar{X}_D$ with cylindrical ends. In what follows, for any fiberwise symplectic 2-form $\omega$ on $\bar{X}_D$ that coincides with $\omega_{X,0}$ outside of some compact subset $K\subset \bar{X}_D$, let $\mathcal{M}_{\omega}(k_\infty;k_1, k_2)$ denote the moduli space of $J$-holomorphic sections of $(\bar{X}_D,\omega)$ whose ends have $x$-coordinates asymptoting to $\frac{k_\infty}{m+n}, \frac{k_1}{m}$, and $\frac{k_2}{n}$, where $k_\infty=k_1+k_2$.

\begin{remark}
We make an observation about $\mathcal{M}_{\omega}(k_\infty;k_1, k_2)$ that is already implicit in its definition. Note here the Reeb orbits come in $S^1$ families, and in defining $\mathcal{M}_{\omega}(k_\infty;k_1, k_2)$ we allow the ends of its elements to land on any Reeb orbit on a given $S^1$ family. In other words, the ends of a $J$-holomorphic section are ``free''. We can also require ends of $J$-holomorphic sections to land on a specific Reeb orbit in a $S^1$ family, in which case the ends are ``fixed''. This distinction will be very important to us when we pass from Morse-Bott case to the Morse case.
\end{remark}

\begin{lemma}\label{shifting lemma}
There is a one-to-one correspondence between $\mathcal{M}_{\omega_{X,0}}(k_\infty;k_1, k_2)$ and $\mathcal{M}_{\omega_{X,0}}(k_\infty+m+n;k_1+m, k_2+n)$.
\end{lemma}
\begin{proof}[Proof of Lemma \ref{shifting lemma}]
Define a diffeomorphism $\mu:\bar{X}_D\to \bar{X}_D$, $(p,x,y)\mapsto (p,x',y')$ by
\begin{equation}
\begin{cases}
      x'=x-1 \\
      y'=y+ g^{m,n}(p) .
    \end{cases}
\end{equation}

Observe that $\mu$ preserves $\Ver$, and a simple calculation shows $\mu^*\omega_{X,0}=\omega_{X,0}$. So $\mu^*J=J$, and hence if $u$ is a $J$-holomorphic section contained in $\mathcal{M}_{\omega_{X,0}}(k_\infty+m+n;k_1+m, k_2+n)$, then $\mu\circ u$ is a $J$-holomorphic section contained in $\mathcal{M}_{\omega_{X,0}}(k_\infty;k_1, k_2)$, and vice versa.
\end{proof}

\begin{lemma} \label{automatic transversality}
For any $u\in\mathcal{M}_{\omega}(k_\infty;k_1, k_2)$, $\ind(u)=1$ and $u$ is cut out transversely.
\end{lemma}
\begin{proof}[Proof of Lemma \ref{automatic transversality}]
Similar to what we did in Section \ref{the product}, choose the trivialization $\tau:\Ver\to T(\R_x\times S^1_y)\cong \R^2$. Let $\alpha$ and $\beta_1,\beta_2$ denote the three ends of $u$. The index formula is:
\begin{align}
\begin{split}
\ind(u)&=1 + 2 \langle c_1^\tau(TX),[C]\rangle+ \CZ^+_\tau(\alpha) - \sum\CZ^-_\tau(\beta_i)\\
    &=-1 + \CZ^+_\tau(\alpha) - \sum \CZ^-_\tau(\beta_i)\\
    &=-1+0-(-1)-(-1)=1.
\end{split}
\end{align}

Now the automatic transversality theorem (\cite{wendl2010automatic}, Theorem 1) applies here, because 
\begin{equation}
1=\ind(u)>c_N(u)+Z(du)=0+0.
\end{equation}
\end{proof}

Notice that Lemma \ref{automatic transversality} doesn't require $\omega$ to be $\omega_{X,0}$. Let us consider a 1-parameter family of closed fiberwise symplectic 2-forms $\omega_\lambda$ for $\lambda\in [0,1]$, where:
\begin{enumerate}
    \item $\omega_0=\omega_{X,0}$,
    \item When restricted to each fiber of $\bar{X}_D\to B_0$, we have $\omega_\lambda|_{\textup{fiber}} = dx\wedge dy$,
    \item There is a compact subset of $\bar{X}_D$ outside of which all $\omega_\lambda$, with $ \lambda\in [0,1]$, agree.
\end{enumerate}
Observe that any closed fiberwise symplectic 2-form $\omega$ that agrees with $\omega_{X,0}$ outside of a compact set and restricts to $dx\wedge dy$ on each fiber can be connected to $\omega_{X,0}$ using a family $\omega_\lambda$ described above: one can simply put $\omega_\lambda=\lambda\omega+(1-\lambda)\omega_{X,0}$. For each $\lambda$, let $J_\lambda$ denote the fibration compatible almost complex structure on $\bar{X}_D$ determined by $\omega_\lambda$. As before, let $\mathcal{M}_{\omega_\lambda}(k_\infty;k_1, k_2)$ denote the moduli space of $J_\lambda$-holomorphic sections of $(\bar{X}_D,\omega_\lambda)$ whose ends have $x$-coordinates $\frac{k_\infty}{m+n}, \frac{k_1}{m}$, $\frac{k_2}{n}$, where $k_\infty=k_1+k_2$.

Let $\{\mathcal{M}_{\omega_\lambda}(k_\infty;k_1, k_2)\}|_{\lambda\in [0,1]}$ denote the parametrized moduli space:
\begin{equation}
\{\mathcal{M}_{\omega_\lambda}(k_\infty;k_1, k_2)\}|_{\lambda\in [0,1]}:=\{(u,\lambda)|u\in\mathcal{M}_{\omega_\lambda}(k_\infty;k_1, k_2), \lambda\in[0,1])\}
\end{equation}

We now describe the specific type of deformed fiberwise symplectic 2-form that we will use. Fix a compact subset $K_1\subset B_0$ such that the complement of $K_1$ is contained in the cylindrical ends. For any $R>0$, denote $[-R,R]_x\times S^1_y\subset \R_x\times S^1_y$ by $Q_R$. 

\begin{definition}\label{admissible}
A 1-form $\sigma\in \Omega^1(B_0, C^\infty(\R\times S^1))$ is called \emph{admissible} if there exists some compact subset $K_2\subset B_0$ containing $K_1$, and some $R>0$, such that $\sigma=\beta^{m,n}\cdot \frac{1}{2}x^2$ outside of $K_2\times Q_{R}$. 
\end{definition}

For any admissible 1-form $\sigma$, we define the corresponding closed fiberwise symplectic 2-form $\omega$ to be $dx\wedge dy+d\sigma$. The following simple observation asserts that if $\omega=dx\wedge dy+d\sigma$ and $\sigma$ is admissible, then all pseudo-holomorphic sections in $\mathcal{M}_{\omega}(k_\infty;k_1, k_2)$ have a uniform upper bound on the vertical energy (see definition below) and the range of its $x$ component is bounded.

\begin{definition}
Fix an admissible 1-form $\sigma$ and the corresponding almost complex structure $J$. Let $g_J$ denote the metric determined by $\omega=d\sigma+dx\wedge dy$ and $J$. For any smooth map $u:B_0\to \bar{X}_D$, define the vertical energy $E(u)$ to be:
\begin{equation}
    E(u)=\frac{1}{2}\int_{B_0}|\partial_s u-\partial_s^\#|^2_{g_J}+|\partial_t u-\partial_t^\#|^2_{g_J} ds\wedge dt
\end{equation}
where $\partial_s^\#$ and $\partial_t^\#$ are the horizontal lifts of the vector fields $\partial_s$ and $\partial_t$, respectively.
\end{definition}
\begin{remark}
Our definition, written in local conformal coordinates, coincides with that of \cite{mcduff2012j} equation 8.1.8. It's also clear from the definition that a smooth map has zero vertical energy if and only if it is a horizontal section.
\end{remark}

\begin{lemma}\label{vertical energy bound}
Let $\sigma_\lambda|_{\lambda\in [0,1]}$ be a family of admissible 1-form as in Definition \ref{admissible} (where $R$ is assumed to be sufficiently large compared to $k_\infty$, $k_1$ and $k_2$), and $\omega_\lambda=dx\wedge dy+d\sigma_\lambda$ be the corresponding closed fiberwise symplectic 2-form on $\bar{X}_D$. Let $J_{\omega_\lambda}$ denote the fibration-compatible almost complex structure determined by $\omega_{\lambda}$. Then for any $u\in \mathcal{M}_{\omega_\lambda}(k_\infty;k_1, k_2)$,
\begin{enumerate}
    \item $u$ is contained in $\{-2R\le x\le 2R\}$, and
    \item The vertical energy $E(u)$ has a uniform bound.
\end{enumerate}
\end{lemma}

\begin{proof}[Proof of Lemma \ref{vertical energy bound}]
Assume that there is some $u\in \mathcal{M}_{\omega_\lambda}(k_\infty;k_1, k_2)$ not contained in $\{-2R\le x\le 2R\}$. It's easy to show that such a section cannot be horizontal (otherwise, inside the region $\{R<|x|<2R\}$ the function $x$ would be locally constant, a contradiction). 
Outside of $\{-2R\le x\le 2R\}$ the 2-form $\omega$ coincides with $\omega_{X,0}$, so after picking some $x_0$ such that $|x_0|>2R$ and $u$ intersects $\{x=x_0\}$ transversely in a non-empty way, the homology class of the intersection $u\bigcap \{x=x_0\}$ satisfies the local energy inequality
\begin{equation}
p+x_0(mq_1+nq_2)>0
\end{equation}
which contradicts the fact that $u\bigcap \{x=x_0\}$ is null-homologous. For the bullet point 2, we first observe that for a sequence of subdomains $D_k$ exhausting $B_0$, we have
\begin{align}
\begin{split}
\int_{B_0}u^*\omega_\lambda &= \lim_{k\to \infty}\int_{D_k}u^*\omega_\lambda\\
                    &=\lim_{k\to \infty}\int_{\partial D_k}u^*(xdy+\sigma_\lambda)\\
                    &=\lim_{k\to \infty}\int_{\partial D_k}u^*(xdy+\beta^{m,n}\cdot \frac{1}{2}x^2).
\end{split}
\end{align}
And the last limit does not depend on $\lambda$ or $u$. We also observe that given the fact $u\in \mathcal{M}_{\omega_\lambda}(k_\infty;k_1, k_2)$, this term is finite.

Next, we calculate $E(u)$ for a $J$-holomorphic section $u$ (again using local conformal coordinates $(s, t)$):
\begin{align}
\begin{split}
    E(u)&=\frac{1}{2}\int_{B_0}|\partial_s u-\partial_s^\#|^2_{g_J}+|\partial_t u-\partial_t^\#|^2_{g_J} ds\wedge dt\\
    &=\int_{B_0}\omega_\lambda(\partial_s u-\partial_s^\#,\partial_t u-\partial_t^\#)ds\wedge dt\\
    &=\int_{B_0}\omega_\lambda(\partial_s u,\partial_t u)-\omega_\lambda(\partial_s^\#,\partial_t^\#)ds\wedge dt \\
    &=\int_{B_0}u^*\omega_\lambda-\int_{B_0}\omega_\lambda(\partial_s^\#,\partial_t^\#)ds\wedge dt
\end{split}
\end{align}

So it suffices to estimate the term $\int_{B_0}\omega_\lambda(\partial_s^\#,\partial_t^\#)ds\wedge dt$. Write $\sigma_\lambda$ as $F^\lambda(s,t,x,y)ds+G^\lambda(s,t,x,y)dt$, and we have:
\begin{equation}
    \omega_\lambda(\partial_s^\#,\partial_t^\#)=\frac{\partial G^\lambda}{\partial s}-\frac{\partial F^\lambda}{\partial t}+\frac{\partial G^\lambda}{\partial x}\frac{\partial F^\lambda}{\partial y}-\frac{\partial G^\lambda}{\partial y}\frac{\partial F^\lambda}{\partial x}
\end{equation}
Now we simply observe that on the union of cylindrical ends $Z$ where $\sigma_\lambda=\frac{1}{2}x^2\beta^{m,n}$, the term $F^\lambda$ is identically zero and $G^\lambda$ is independent of $s$. Consequently, $\omega_\lambda(\partial_s^\#,\partial_t^\#)$ is compactly supported on
\begin{equation}
    (s,t,x,y,\lambda)\in (B_0-Z)\times [-2R,2R]\times S^1\times [0,1]
\end{equation}
and hence has a uniform bound for all $u\in \mathcal{M}_{\omega_\lambda}(k_\infty;k_1, k_2)$.

\end{proof}

Lemma \ref{vertical energy bound}, together with the usual SFT compactness argument developed in \cite{bourgeois2003compactness} (see also \cite{wendl2016lectures} for a nice account), tells us that for any closed fiberwise symplectic 2-form $\omega=dx\wedge dy+d\sigma$ where $\sigma$ is admissible, the moduli space $\mathcal{M}_{\omega}(k_\infty;k_1, k_2)$ is compact. To see that no SFT type breaking can occur for a sequence of sections in $\mathcal{M}_{\omega}(k_\infty;k_1, k_2)$, we observe that levels in the symplectizations are necessarily cylinders, and such cylinders asymptote to orbits in the same Morse-Bott family for homological reasons. Now such cylinders have zero vertical energy, hence are trivial cylinders. We also observe no bubbles appear, as $\pi_2$ of the bundle is trivial.

More generally, for such $\omega$, if we define the 1-parameter family $\omega_\lambda:=\lambda\omega+(1-\lambda)\omega_{X,0}$, then the parametric moduli space $\{\mathcal{M}_{\omega_\lambda}(k_\infty;k_1, k_2)\}|_{\lambda\in [0,1]}$ is compact as well. 
Lemma \ref{automatic transversality} tells us that for such $\omega$, both $\mathcal{M}_{\omega}(k_\infty;k_1, k_2)$ and $\{\mathcal{M}_{\omega_\lambda}(k_\infty;k_1, k_2)\}|_{\lambda\in [0,1]}$ are transversely cut out, so $\{\mathcal{M}_{\omega_\lambda}(k_\infty;k_1, k_2)\}|_{\lambda\in [0,1]}$ is a compact cobordism between two closed 1-dimensional manifolds. Observe also that for each fixed $\lambda' \in [0,1]$, the slice $\mathcal{M}_{\omega_{\lambda'}}(k_\infty;k_1, k_2) \subset \{\mathcal{M}_{\omega_\lambda}(k_\infty;k_1, k_2)\}|_{\lambda\in [0,1]}$ is a closed 1-manifold, so we get the following:

\begin{cor}\label{same moduli}
For any $\omega=dx\wedge dy+d\sigma$ where $\sigma$ is admissible, $\mathcal{M}_{\omega}(k_\infty;k_1, k_2)$ is diffeomorphic to $\mathcal{M}_{\omega_{X,0}}(k_\infty;k_1, k_2)$.
\end{cor}

The next observation (Corollary \ref{uniqueness cor}) asserts that $\mathcal{M}_{\omega_{X,0}}(k_\infty;k_1, k_2)$ has at most one component. To get started, let us observe that there is an $S^1$ symmetry of $\mathcal{M}_{\omega_{X,0}}(k_\infty;k_1, k_2)$. In the following, we will view $J$-holomorphic sections $u$ of $(\bar{X}_D,\omega)$ as maps $\bar{u}:B_0\to \R\times S^1$, so it is handy to establish the following Lemma:

\begin{lemma}\label{dictionary to 2d}
Choose a local conformal coordinate $(s,t)$ of $B_0$ and suppose locally 
\[
\sigma =\frac{1}{2}x^2(F(s,t)ds+G(s,t)dt)
\]
\[
\omega_{X,0}=dx\wedge dy+d\sigma.
\]
A map $u:B_0\to (\bar{X}_D,\omega)$ is a
$J$-holomorphic section if and only if the corresponding map $\bar{u}=(x(s,t),y(s,t)):B_0\to \R\times S^1$ in our coordinate system solves the following PDE:
\begin{equation}
\begin{cases}
\frac{\partial x}{\partial t}+\frac{\partial y}{\partial s}+xF=0\\
\frac{\partial y}{\partial t}-\frac{\partial x}{\partial s}+xG=0.
\end{cases}
\end{equation}
\end{lemma}

\begin{proof}[Proof of Lemma \ref{dictionary to 2d}]
Let $v^\#$ denote the horizontal lift (with respect to $\omega$) of any vector $v\in TB_0$. A simple calculation shows that:
\begin{equation}
\partial_s^\#=\partial_s-xF\partial_y,\quad \partial_t^\#=\partial_t-xG\partial_y.
\end{equation}
So by definition, $J(\partial_s-xF\partial_y)=\partial_t-xG\partial_y$. Recall that we required that $J$ always sends $\partial_x$ to $\partial_y$, so this shows:
\begin{equation}
J(\partial_s)=\partial_t-xF\partial_x-xG\partial_y.
\end{equation}
Now suppose $u:B_0\to \bar{X}_D$, $(s,t)\mapsto (s,t,x,y)$ is $J$-holomorphic, i.e.
\begin{equation}
J(\partial_s+\frac{\partial x}{\partial s}\partial_x+\frac{\partial y}{\partial s}\partial_y)=\partial_t+\frac{\partial x}{\partial t}\partial_x+\frac{\partial y}{\partial t}\partial_y.
\end{equation}
Combine the above equations and collect the coefficients of $\partial_x$ and $\partial_y$, and we get the desired PDE.
\end{proof}

The following corollary is immediate:
\begin{cor}\label{S1 symmetry}
If $u:B_0\to \bar{X}_D$ is a map given by $(s,t)\mapsto (s,t,x(s,t),y(s,t))$, and $u$ is an element of $\mathcal{M}_{\omega_{X,0}}(k_\infty;k_1, k_2)$, then for any $y_0\in S^1$, $(s,t)\mapsto (s,t,x(s,t),y(s,t)+y_0)$ is also an element of $\mathcal{M}_{\omega_{X,0}}(k_\infty;k_1, k_2)$.
\end{cor}

In other words, if $\mathcal{M}_{\omega_{X,0}}(k_\infty;k_1, k_2)\ne \emptyset$, then $S^1$ acts freely on $\mathcal{M}_{\omega_{X,0}}(k_\infty;k_1, k_2)$ by translating the $y$-coordinate. We now show that this $S^1$ action is also transitive.

\begin{lemma}\label{uniqueness}
Suppose $u_1$ and $u_2$ are two different $J$-holomorphic sections in $\mathcal{M}_{\omega_{X,0}}(k_\infty;k_1, k_2)$, and let $\bar{u}_i:B_0\to \R\times S^1$ denote the corresponding maps to the fiber, which we write as $(s,t)\mapsto (x_i,y_i)$. Then $x_1=x_2$, and there is some $y_0\in S^1$ such that $y_2(s,t)=y_1(s,t)+y_0$.
\end{lemma}

\begin{proof}[Proof of Lemma \ref{uniqueness}]
Consider $w:=\bar{u}_1-\bar{u}_2:B_0\to \R\times S^1$, $(s,t)\mapsto (l(s,t), m(s,t))$. By Lemma \ref{dictionary to 2d}, locally the following PDEs are satisfied:
\begin{equation}
\begin{cases}
\frac{\partial l}{\partial t}+\frac{\partial m}{\partial s}+lF=0\\
\frac{\partial m}{\partial t}-\frac{\partial l}{\partial s}+lG=0.
\end{cases}
\end{equation}

Observe that by our assumption, $u_1$ and $u_2$ have the same asymptotics on their cylindrical ends, so it follows that $w$ induces the trivial map on $\pi_1$, hence can be lifted to $\Tilde{w}:B_0\to \R\times \R$, $(s,t)\mapsto (l(s,t), \Tilde{m}(s,t))$, where $(l,\Tilde{m})$ solves the same PDE. Let us denote the covering map $\R\times \R\to \R\times S^1$ by $\pi$.

Now for any $\zeta\in\R$, the map $\bar{u}_2+\pi(\zeta\cdot \Tilde{w})$ solves the same PDE, hence gives an element in $\mathcal{M}_{\omega_{X,0}}(k_\infty;k_1, k_2)$. This implies that $l(s,t)$ is identically zero and hence $\Tilde{m}(s,t)$ is constant, because otherwise,
\begin{equation}
\bar{u}_2+\pi(\zeta\cdot \Tilde{w}+\xi \cdot (0,1))
\end{equation}
solves the same PDE for any $\zeta,\xi\in\R$, giving us a two dimensional family of solutions. This contradicts the fact that $u_2$ is cut out transversely and $\ind(u_2)=1$.
\end{proof}

A corollary of the above discussion is the following:

\begin{cor}\label{uniqueness cor}
For any admissible $\sigma$, the moduli space $\mathcal{M}_{dx\wedge dy+d\sigma}(k_\infty;k_1, k_2)$ is either empty, or diffeomorphic to $S^1$.
\end{cor}

We now proceed to show that $\mathcal{M}_{dx\wedge dy+d\sigma}(k_\infty;k_1, k_2)$ is not empty. By Lemma \ref{shifting lemma}, we can assume that $k_\infty$, $k_1$ and $k_2$ are all positive. In light of Corollary \ref{same moduli}, it suffices to find \emph{one} special admissible $\sigma$, and show that $\mathcal{M}_{\omega}(k_\infty;k_1, k_2)$ is not empty, where $\omega=dx\wedge dy+d\sigma$. The rest of this subsection describes how one can construct an admissible $\sigma$ with a nonempty $\mathcal{M}_{dx\wedge dy+d\sigma}(k_\infty;k_1, k_2)$.

To this end, we make the following observation. Fix cylindrical ends $Z_\infty$, $Z_1$ and $Z_2$ of $B_0$ outside of $K_1$, each with conformal coordinates $(s_i,t_i)\in [N,\infty)\times S^1)$ (or $(-\infty, -N]\times S^1$), and choose cutoff functions $\chi_i:Z_i\to [0,1]$ such that $1-\chi_i$ are compactly supported. Suppose $v:B_0\to \R\times S^1$ is a smooth map, such that (notice the resemblance to the formulations in \cite{schwarz1995cohomology}):
\begin{enumerate}
    \item The restriction of $v=(x,y)$ to the cylindrical end $Z_\infty$ solves the equation:
    \begin{equation}
    \begin{cases}
    \frac{\partial x}{\partial t_\infty}+\frac{\partial y}{\partial s_\infty}=0\\
    \frac{\partial y}{\partial t_\infty}-\frac{\partial x}{\partial s_\infty}+(m+n)\chi_\infty(s_\infty,t_\infty)x=0
    \end{cases}
    \end{equation}
    \item The restriction of $v=(x,y)$ to the cylindrical end $Z_1$ solves the equation:
    \begin{equation}
    \begin{cases}
    \frac{\partial x}{\partial t_1}+\frac{\partial y}{\partial s_1}=0\\
    \frac{\partial y}{\partial t_1}-\frac{\partial x}{\partial s_1}+m\chi_1(s_1,t_1)x=0
    \end{cases}
    \end{equation}
    \item The restriction of $v=(x,y)$ to the cylindrical end $Z_2$ solves the equation:
    \begin{equation}
    \begin{cases}
    \frac{\partial x}{\partial t_2}+\frac{\partial y}{\partial s_2}=0\\
    \frac{\partial y}{\partial t_2}-\frac{\partial x}{\partial s_2}+n\chi_2(s_2,t_2)x=0
    \end{cases}
    \end{equation}
    \item The restriction of $v=(x,y)$ to the complement of $Z_1\bigcup Z_2\bigcup Z_\infty$ is holomorphic:
    \begin{equation}
    \frac{\partial x}{\partial t}+\frac{\partial y}{\partial s}=\frac{\partial y}{\partial t}-\frac{\partial x}{\partial s}=0
    \end{equation}
    \item The map $v$ approaches the projections (to the fiber $S^1\times \R$) of three Reeb orbits at $x_1=\frac{k_1}{m}, x_2=\frac{k_2}{n}, x_\infty =\frac{k_\infty}{m+n}$ at its corresponding cylindrical ends. Note the Reeb orbits whose projections the map $v$ approaches all live in $S^1$ families. We do not care which orbits in these $S^1$ families $v$ approaches.
\end{enumerate}

Then we can construct an admissible $\sigma$ such that the moduli space $$\mathcal{M}_{dx\wedge dy+d\sigma}(k_\infty;k_1, k_2)$$ is not empty. The reason is that $\Tilde{v}:(s,t)\mapsto (s,t,v(s,t))$ is a smooth section of $\bar{X}_D\to B_0$ with the desired asymptotes, and the image of $v$ is contained in $Q_R$ for some large $R$. If we define an admissible 1-form $\sigma$ such that
\begin{equation}
\sigma = \begin{cases}
         \frac{m+n}{2}x^2 \chi_\infty(s_\infty,t_\infty)dt_\infty,\quad &\text{in}\ Z_\infty\times Q_R \\
         \frac{m}{2}x^2 \chi_1(s_1,t_1)dt_1,\quad &\text{in}\ Z_1\times Q_R \\
         \frac{n}{2}x^2 \chi_2(s_2,t_2)dt_2,\quad &\text{in}\ Z_2\times Q_R \\
         0, &\text{in}\ (B_0-Z_1\bigcup Z_2\bigcup Z_\infty)\times Q_R
         \end{cases}
\end{equation}
Then Lemma \ref{dictionary to 2d} tells us that $\Tilde{v}\in \mathcal{M}_{dx\wedge dy+d\sigma}(k_\infty;k_1, k_2)$.

We now construct $v$ as described above. In the following, inside the cylindrical ends $Z_i$, we will always assume that $\chi_i$'s are $t_i$-independent and monotone, and that:
\begin{equation}
y=\begin{cases}
   -k_\infty t_\infty,\ &\text{in}\ Z_\infty\\
   -k_1 t_1,\ &\text{in}\ Z_1\\
   -k_2 t_2,\ &\text{in}\ Z_2
  \end{cases}
\end{equation}

So inside $Z_i$, $x$ is $t_i$-independent, and the equations simplify to ODE's
\begin{equation}
\frac{\partial x}{\partial s_\infty}+k_\infty-(m+n)\chi_\infty(s_\infty) x=0
\end{equation}
\begin{equation}
\frac{\partial x}{\partial s_1}+k_\infty-m\chi_1(s_1) x=0
\end{equation}
\begin{equation}
\frac{\partial x}{\partial s_2}+k_\infty-n\chi_2(s_2) x=0.
\end{equation}
If we let 
\begin{equation}
a_\infty(s_\infty):=\int_N^{s_\infty}-(m+n)\chi_\infty(\rho)d\rho
\end{equation}
\begin{equation}
a_1(s_1):=\int_{-N}^{s_1}-m\chi_1(\rho)d\rho
\end{equation}
and 
\begin{equation}
a_2(s_2):=\int_{-N}^{s_2}-n\chi_2(\rho)d\rho,
\end{equation}
then the solutions to the above ODEs can be explicitly written down: 
\begin{equation}x(s_\infty)=e^{-a_\infty}(c_\infty-k_\infty\int_N^{s_\infty}e^{a_\infty(\rho)}d\rho)\end{equation}
\begin{equation}x(s_1)=e^{-a_1}(c_1-k_1\int_{-N}^{s_1}e^{a_1(\rho)}d\rho)\end{equation}
\begin{equation}x(s_2)=e^{-a_2}(c_2-k_2\int_{-N}^{s_2}e^{a_2(\rho)}d\rho).
\end{equation}

Based on our assumptions, it's easy to show that for the first expression, there is a unique choice of $c_\infty\in\R$ such that $x(s_\infty)$ stays at $\frac{k_\infty}{m+n}$ when $s_\infty$ is large enough. On the other hand, whatever choice of $c_1$ or $c_2$ is, the second(resp. the third expression) will converge to $\frac{k_1}{m}$ (resp. $\frac{k_2}{n}$) when $s_i$ close enough to $-\infty$. Notice also that near the boundaries of $Z_i$ (i.e. when $s_i$ is close to $\pm N$) the functions $x(s_i)$ are linear with slopes $-k_i$.

For a suitable choice of $c_\infty<c_1=c_2$, it is not hard to construct a holomorphic $k_\infty$-fold branched cover $v=(x,y):B_0-Z_1\bigcup Z_2\bigcup Z_\infty\to [c_\infty, c_1]\times S^1$ such that near $\partial Z_\infty$ (resp. $\partial Z_1$, $\partial Z_2$), $(x,y)=(c_\infty-k_\infty(s_\infty-N), -k_\infty t_\infty)$ (resp. $(c_1-k_1(s_1+N), -k_1 t_1)$ and $(c_2-k_2(s_2+N), -k_2 t_2)$). So for such choices of $c_\infty<c_1=c_2$, the map $v$ glues smoothly with the three solutions on $Z_i$'s. The above discussion shows the following:
\begin{proposition}\label{morse-bott moduli space}
For any admissible $\sigma$, $\mathcal{M}_{dx\wedge dy+d\sigma}(k_\infty;k_1, k_2)$ is diffeomorphic to $S^1$.
\end{proposition}

\subsection{Morse-Bott theory and enumeration of sections after the perturbation}\label{subsection: MB}
In this subsection we use some Morse-Bott theory to enumerate sections after we replace $\omega_{X,0}$ with $\omega_X$ by adding a small Hamiltonian perturbation that breaks the Morse-Bott degeneracy (see Section \ref{setup}). We know from Theorem \ref{cobordism eta=0} and Theorem \ref{classification for general J} that no holomorphic sections cross from $X_H$  to $X_D$ or vice versa, so we only need to consider what happens to $J$-holomorphic sections that stay in $X_D$ as we break the Morse-Bott degeneracy.

For simplicity let us work instead in $\bar{X}_D$. By abuse of notation, we will use the same letter $\omega_X$ to denote the 2-form on the completion $\bar{X}_D$, and likewise for $\omega_{X,0}$. 
We first recall some conventions for $J$-holomorphic sections whose ends land on Morse-Bott submanifolds. We fix $J$ the fibration-compatible almost complex structure (we will have a bit more to say about choice of $J$ after we describe cascades). In our case all Reeb orbits come in $S^1$ families, corresponding to $x=\frac{k_1}{m},\frac{k_2}{m},\frac{k_1+k_2}{m+n}$ in $\bar{X}_D$. Hence we have tori that are foliated by Reeb orbits (we will call such tori ``Morse-Bott tori'').

Recall for $J$-holomorphic sections ending on Morse-Bott tori, we can consider moduli spaces of sections with ``fixed'' end points and moduli spaces of sections with ``free'' end points. These are conditions we impose on the given cylindrical ends of the holomorphic section. ``Fixed'' end points condition means the given end of the section must land on a specific Reeb orbit in the Morse-Bott torus, whereas the ``free'' end condition means that the given end of a section in the moduli space is allowed to freely move around on the Morse-Bott torus. The dimension of the moduli space of course depends on how many ends are specified as fixed or free. For instance the moduli space we constructed in the previous section, $\mathcal{M}_{dx\wedge dy +d\sigma} (k_\infty;k_1,k_2)$ has all ends free by this convention (we implicitly assumed this in our previous construction).

We shall invoke some standard Morse-Bott theory (which is easily adapted to cobordisms) to figure out how to count our $J$-holomorphic sections. However, the statement of the entire theory is rather cumbersome, instead we will just state the small parts we need. We refer the reader to \cite{YY} for a more detailed account.

We recall that in breaking the Morse-Bott degeneracy via perturbations, each torus of Morse Bott orbits breaks into an hyperbolic orbit $h_k$ at maximum of $h(y)$ and an elliptic orbit $e_k$ at minimum of $h(y)$. These orbits form the generators of the fixed point Floer homology.

For our purposes,  $J$-holomorphic sections of degree one and  Fredholm index zero in  $(\bar{X}_D,\omega_X)$ which we need to count correspond to cascades in $(\bar{X}_D,\omega_{X,0})$. The cascades, generally speaking, take the following form:
\begin{itemize}
    \item There is a main level $u_0: B_0 \rightarrow (\bar{X}_D,\omega_{X,0})$;
    
    \item There are upper levels indexed by $i\in \Z_{>0}$. We write them as $u_{i>0}: S^1 \times \R \rightarrow (S^1 \times \R\times S^1 \times \R,\omega_{X,0})$;
    
    \item There are lower levels indexed by $j\in\mathbb{Z}_{<0}$ labelling the level number, and $k \in \{1,2\}$ labeling which negative puncture  it corresponds to. We write them as:
    $u^{k}_{j<0}: S^1 \times \mathbb{R} \rightarrow (S^1 \times \R\times S^1 \times \R,\omega_{X,0})$.

    \item For each map $u_i$ let $\pi$ denote the projection to the base in the codomain, then $\pi \circ u_i$ is the identity; likewise for $u_j^k$;
    \item Let $ev^\pm(u_i)$ denote the Reeb orbit $u_i$ approaches as $s\rightarrow \pm \infty$, let $\phi_T$ denote the gradient flow of $h(y)$ for time $T$, then there exists $T_i \in (0,\infty)$ so that $\phi_{ T_i}(ev^-(u_i))=ev^+(u_{i-1})$;
    \item For the main level we have numbers $T_1, T_k^0 \in (0,\infty)$ for $ k \in \{1,2\}$ so that:
    \[
    \phi_{ T_1}(ev^-(u_1))=ev^+(u_{0})
    \]
    \[
    \phi_{T^{k}_0}(ev^-(u_0))=ev^+(u^{k}_{-1}).
    \]
    \item For the lower levels, there are numbers $T_j^k \in (0,\infty)$ for $k\in \{1,2\}$ so that $\phi_{T_j^k}(ev^-(u_j^k))=ev^+(u^k_{j-1})$.
\end{itemize}

Generally speaking there are more conditions we can achieve for cascades by choosing generic $J$, however for our case we immediately observe for homological reasons $u_{i\neq0}$ are cylinders, in fact they must all be trivial cylinders by energy considerations, so there is only the main level, which for ease of notation we denote by $u$. Since our entire cascade only has one level, in order to count $u$, it must live in a moduli space of dimension zero.

Thus we arrive at the following description of the cascades we must count:
\begin{proposition} \label{cas}
The (1-level) cascades that we need to count takes the following form:
\begin{itemize}
    \item $u:B_0 \rightarrow (\bar{X}_D,\omega_{X,0})$ is a $J$-holomorphic section.
    \item Let the ends of $B_0$ be labelled $\{1,2,\infty\}$. Then one of the ends in  $\{1,2,\infty\}$ is fixed, the other 2 are free. Hence $u$ belongs in a (transversely cut out) moduli space of index zero. (Fixing one end reduces the virtual dimension by one).
    \item All free ends avoid critical points of $h(y)$ (this can be achieved for generic $J$). If the ends labelled $1,2$ are fixed, then they land on the maximum of $h(y)$. If the $\infty$ end is fixed, it lands on the minimum of $h(y)$.
\end{itemize}
\end{proposition}
That the cascades above are the ones we need to count to compute the co-product map is supplied by the following correspondence theorem:
\begin{proposition}\label{proposition:correspondence}
Given a perfect Morse function $h(y)$, we make the Morse perturbation smaller by rescaling it to be $\bar{\delta} h(y)$, where $\bar{\delta}>0$ is a small positive real number. For small enough $\bar{\delta} >0$, there is a 1-1 correspondence between $J$-holomorphic sections in the nondegenerate case and $J$ holomorphic cascades. Given a cascade $u$ of the form described in the previous proposition, for the positive end of $u$, if it is a free end we assign the generator $h_k$, and if it a fixed end we assign it the generator $e_k$. We reverse the assignments for the two negative ends $1,2$. Then each cascade gives rise bijectively to one $J$ holomorphic section beginning and ending at the generators we assigned to the ends of cascade.
\end{proposition}

\begin{proof}
See \cite{YY} for a proof of this statement and how the correspondence works in the general case of multiple level cascades. However since we are only working in the simple case of 1-level cascades, the correspondence theorem that we need is also established in the Appendix of \cite{colin2022embedded} cowritten with Yao.
\end{proof}
\begin{remark}
Note by the above correspondence theorem, the only requirement we need to impose on $J$ in the Morse-Bott case is that all moduli spaces of index zero sections listed above are transversely cut out and that the free ends avoid critical points of $h(y)$. It will be apparent from the paragraph following this remark that the fibration compatible $J$ we chose in Section \ref{coproduct curves}, with the help of automatic transversality, suffices for the purpose of showing the index zero sections are transversely cut out. To ensure the free ends avoid critical points of $h$, instead of further perturbing $J$, we shall instead choose generic $h$. This kind of strategy was also undertaken when Morse-Bott techniques were employed in \cite{ECHT3}.

We also remark that in the correspondence between cascades and $J$ holomorphic sections we need to perturb the almost complex structure from $J$ in the Morse-Bott case in $\bar{X}_D$ to a generic $J_{\bar{\delta}}$ in the non-degenerate case to establish a correspondence between cascades and holomorphic sections, see \cite{YY} for more details. The difference between $J_{\bar{\delta}}$ and $J$ can be taken to be $C^\infty$ small. Hence for small enough $\bar{\delta}>0$ we do not need to worry about this change in almost complex structure since we already established in Section \ref{general no crossing} that for $C^\infty$ small perturbations of $J$ the no-crossing results continues to hold.
\end{remark}
Recall that Proposition \ref{morse-bott moduli space} tells us the moduli space of sections with all three ends free in the Morse-Bott case come in $S^1$ families, and this $S^1$ family is precisely rotation around the $\partial_y$ direction, i.e. rotation along the Morse-Bott tori. To obtain a cascade of Fredholm index zero in the form specified above, it is readily apparent we restrict one of the three ends fixed and the rest two free. All of our cascades we need arise this way.

Instead of perturbing $J$ further to ensure that when we have fixed an end, all the remaining free ends avoid critical points of $h$, we choose generic Morse perturbations $h$. To be specific, we choose perfect Morse functions $\{h_l^i(y)\}$, where $l\in \{1,2,\infty\}$ labels which cylindrical end of the base we are referring to, and $i \in \mathbb{Z}$ refers to the specific Morse-Bott torus in that end. (For example when $l=1$ and $i=1$, this refers to the Morse Bott torus at $x=1$ for the negative end labelled by $k=1$). We then use these functions to break the Morse-Bott degeneracy.

Given some large integer $N$, for generic choice of $\{h_l^i(y)\}$ we can arrange that for tuples $(k_1,k_2,k_\infty)$ satisfying $0\leq k_1,k_2,k_\infty \leq N$ and $k_\infty=k_1+k_2$, if we consider the moduli space $\mathcal{M}_{\omega_{X,0}}(k_\infty;k_1, k_2)$ (here all ends are free) which is diffeomorphic to $S^1$, if any element in this $S^1$ family has one end landing on a critical point of any of the Morse functions in $\{h_l^i(y)\}$, then all the other ends avoid critical points of elements in $\{h_l^i(y)\}$. We can arrange for this to happen by picking generic collection $\{h_l^i(y)\}$ because we only need to consider finite number of moduli spaces.

It is then apparent if a section in the $S^1$ family $\mathcal{M}_{\omega_{X,0}}(k_\infty;k_1, k_2)$ has one end hitting a critical point, it is the only section in this $S^1$ family with an end on that critical point. 
This, combined with the cascade and holomorphic section correspondence (Proposition \ref{proposition:correspondence}) thus produces the computation required for the coproduct structure restricted to $\bar{X}_D$.
\begin{remark}
Even though all our previous proofs (e.g. no-crossing) assumed we used a single Morse function $h(y)$ to break the Morse-Bott degeneracy, one can check that using a collection of Morse functions $\{h_l^i(y)\}$ to break the degeneracy does not make a difference to our previous results.
\end{remark}

Now we are ready to compute the coproduct map induced by $X^{m,n}$. Let $h^{m+n}_{k_\infty}$ (resp. $e^{m+n}_{k_\infty}$) be the hyperbolic (resp. elliptic) orbit over $x=\frac{k_\infty}{m+n}$ at the positive end, $h^m_{k_i}$ (resp. $e^m_{k_i}$) be the hyperbolic (resp. elliptic) orbit over $x=\frac{k_i}{m}$ at the first negative end, and $h^n_{k_i}$ (resp. $e^n_{k_i}$) be the hyperbolic (resp. elliptic) orbit over $x=\frac{k_i}{n}$ at the second negative end. Let $\Delta$ be the coproduct map induced by $(X^{m,n},\omega_X)$. The above discussion shows the following:

\begin{cor}\label{coproduct result in dehn}

\begin{equation}\label{coproduct for e}
\Delta([e^{m+n}_{k_\infty}])=\sum_{k_i\in\{0,1,\cdots,m\},\ k_\infty-k_i\in\{0,1,\cdots,n\}} [e^m_{k_i}]\otimes [e^n_{k_\infty-k_i}]
\end{equation}
\begin{equation}\label{coproduct for h}
    \Delta([h^{m+n}_{k_\infty}])=\sum_{k_i\in\{0,1,\cdots,m\},\ k_\infty-k_i\in\{0,1,\cdots,n\}} [e^m_{k_i}]\otimes [h^n_{k_\infty-k_i}]+[h^m_{i}]\otimes [e^n_{k_\infty-k_i}]
\end{equation}
For all $k_\infty\in\{0,1,\cdots,m+n\}$.
\end{cor}
\begin{remark}
Notice that in equations (\ref{coproduct for e}) and (\ref{coproduct for h}), the homology classes $[e^{m+n}_0]$, $[h^{m+n}_0]$, $[e^{m+n}_{m+n}]$, $[h^{m+n}_{m+n}]$, $[e^m_0]$, $[h^m_0]$, $[e^m_m]$, $[h^m_m]$, $[e^n_0]$, $[h^n_0]$, $[e^n_n]$, $[h^n_n]$ refer to homology classes in $H_*(\Sigma_0;\Z_2)$ in the sense of the decomposition (\ref{HF isomorphism}). It's also not difficult to see that the computations for $\Delta([e^{m+n}_0]$, $\Delta([h^{m+n}_0])$, $\Delta([e^{m+n}_{m+n}])$ and $\Delta([h^{m+n}_{m+n}])$ coincide with that of the coproduct structure $\Delta_0$ on $H_*(\Sigma_0;\Z_2)$.
\end{remark}

Finally, We are able to complete the proof of Theorem \ref{coproduct for multiple twists}.
\begin{proof}[Proof of Theorem \ref{coproduct for multiple twists}]
Let $J$ be an almost complex structure on $X^{m,n}$ sufficiently close to a fibration-compatible one in the sense of Theorem \ref{classification for general J}. By Corollary \ref{coproduct sections classification}, all $J$-holomorphic sections that are counted in the cobordism map are either contained in the twist region $X_D$ or $X_H$. Similar to what we observed in the proof of Theorem \ref{product for multiple twists}, the count of $J$-holomorphic sections contained in $X_H$ precisely corresponds to the coproduct
\begin{equation}\label{easy part of coproduct}
    HF_*(\phi^{m+n}) \supset H_*(\Sigma_0;\Z_2)\xrightarrow{\quad\Delta_0\quad}H_*(\Sigma_0;\Z_2)\otimes H_*(\Sigma_0;\Z_2)\longrightarrow HF_*(\phi^{m})\otimes HF_*(\phi^{n})
\end{equation}
in the sense of the decomposition (\ref{HF isomorphism}). The remaining parts of Theorem \ref{coproduct for multiple twists} readily follows from Corollary \ref{coproduct result in dehn}.

\end{proof}

\printbibliography
\end{document}